\documentclass[reqno,12pt]{amsart}
\usepackage{amsfonts}
\usepackage{a4wide}	
\usepackage{bm}

\numberwithin{equation}{section}

\usepackage{amsmath,amssymb,amsthm,amsfonts}
\usepackage{mathrsfs}
\usepackage{bbm}
\usepackage{hyperref}
\usepackage{xcolor}
\usepackage{extarrows}

\allowdisplaybreaks[4]%

\newtheorem{lemma}{Lemma}[section]
\newtheorem{theorem}{Theorem}[section]
\newtheorem{definition}{Definition}[section]
\newtheorem{proposition}{Proposition}[section]
\newtheorem{remark}{Remark}[section]

\newcommand{\R}{\mathbb{R}}

\renewcommand{\S}{\mathbb{S}}

\newcommand{\CT}{\mathcal{T}}

\newcommand{\ga}{\gamma}

\newcommand{\eps}{\epsilon}

\newcommand{\vps}{\varepsilon}

\newcommand{\norm}[1]{\left\Vert #1 \right\Vert}
\newcommand{\moc}[1]{\left\lvert #1 \right\rvert}

\date{}							

\everymath{\displaystyle}
\begin{document}
	\title[Low-regularity well-posedness for the Boltzmann equation near vacuum]{Low-regularity global well-posedness for the Boltzmann equation near vacuum}

\author[X.-F. Hu]{Xinfeng Hu}
\address[XFH]{School of Mathematics Science, Fudan University, Shanghai 200433, P.R. ~China}
\email{xfhu25@m.fudan.edu.cn}

\author[S.-Q. Liu]{Shuangqian Liu*}
\address[SQL]{School of Mathematics and Statistics, and Key Lab NAA-MOE, Central China Normal University, Wuhan 430079, P.R.~China}
\email{sqliu@ccnu.edu.cn}
\thanks{*Corresponding author}

\author[H.-R. Peng]{Haoran Peng}
\address[HRP]{School of Mathematics and Statistics, and Hubei Key Lab--Math. Sci., Central China Normal University, Wuhan 430079, P.R.~China}
\email{mathematicphr@mails.ccnu.edu.cn}

\author[Y. Zhou]{Yi Zhou}
\address[YZ]{School of Mathematics Science, Fudan University, Shanghai 200433, P.R. ~China}
\email{yizhou@fudan.edu.cn}

\date{\today}



	\maketitle

	\begin{abstract}
        We study the Boltzmann equation near vacuum in anisotropic low-regularity Besov spaces. We establish the global existence and uniqueness of strong solutions with the critical regularity index $2/p$ for $p\in[1,\infty)$ in $\R^3$. The proof relies on a new bilinear estimate for the nonlinear collision operator. Combined with a div-curl type lemma we develop, this allows us to close the {\it a priori} estimates and thereby obtain global well-posedness.

	\end{abstract}

	
\tableofcontents	
    \section{Introduction}

\subsection{The problem}
The three-dimensional Boltzmann equation, which describes the time evolution of the velocity distribution function $\tilde{f}= \tilde{f}(t,\tilde{x},\tilde{v})\geq 0$ of particles at position $\tilde{x} = (\tilde{x}_1, \tilde{x}_2,\tilde{x}_3)\in \mathbb{R}^3$, velocity $\tilde{v} = (\tilde{v}_1,\tilde{v}_2,\tilde{v}_3)\in \mathbb{R}^3$, and time $t\geq 0,$ reads
    \begin{equation}\label{Bolzmann Equation}
          \left\{
        \begin{array}{l}
          \partial_t \tilde{f}(t,\tilde{x},\tilde{v}) + \tilde{v}\cdot \nabla_{\tilde{x}} \tilde{f}(t,\tilde{x},\tilde{v})=\tilde{Q}(\tilde{f},\tilde{f}),\\[2ex]
            \tilde{f}(0,\tilde{x},\tilde{v})=\tilde{f}_0(\tilde{x},\tilde{v}),
        \end{array}
        \right.
        \end{equation}
        where the initial data $\tilde{f}(0,\tilde{x},\tilde{v}) = \tilde{f}_0(\tilde{x},\tilde{v})$ is given. The bilinear Boltzmann collision operator $\tilde{Q}(\cdot,\cdot)$ is defined as
        \begin{align}\label{con-op}
         \tilde{Q}( \tilde{f},\tilde{g} )=&\int_{\mathbb{R}^3\times \mathbb{S}^2} \moc{\tilde{u}-\tilde{v}}^{\gamma}\mathbf{b}(\cos \tilde{\theta})\tilde{f}(\tilde{v}')\tilde{g}(\tilde{u}')d \tilde{u} d\tilde\omega-\int_{\mathbb{R}^3\times \mathbb{S}^2} \moc{\tilde{u}-\tilde{v}}^{\gamma}\mathbf{b}(\cos \tilde{\theta})\tilde{f}(\tilde{v})\tilde{g}(\tilde{u})d \tilde{u} d\tilde\omega.
        \end{align}
        Here, the pre- and post-collision velocities are related by
        $$ \tilde{u}' = \tilde{u} + \tilde\omega\cdot(\tilde{v}-\tilde{u})\tilde\omega,\  \tilde{v}'=\tilde{v} - \tilde\omega\cdot(\tilde{v}-\tilde{u})\tilde\omega. $$
        and the deviation angle $\tilde{\theta}$ satisfies $\cos \tilde{\theta} :=\dfrac{\tilde{v}-\tilde{u}}{\moc{\tilde{v}-\tilde{u}}}\cdot \tilde\omega,$ under the Grad's angular cutoff assumption $$ 0\leq \mathbf{b}(\cos \tilde{\theta})\leq C\moc{\cos \tilde{\theta}}.$$ The collision frequency index $\ga\in(-3,1]$, characterizes the potential type, with $\ga\in[0,1]$ corresponding to hard potentials and $\ga\in(-3,0)$ corresponding to soft potentials.

To apply the div-curl lemma \ref{div-curl Lemma} and to balance the relative velocity in the collision kernel of \eqref{con-op}, we perform a change of variables to reformulate \eqref{Bolzmann Equation}. For fixed $\tilde{v},\tilde{u}$, denote $e = \dfrac{\tilde{v}-\tilde{u}}{\moc{\tilde{v}-\tilde{u}}}$, and let $A$ be an orthogonal matrix such that
$x=(\eta,y_1,y_2) = A\tilde{x}$ with $\eta=\tilde{x}\cdot e$. With this choice of $A$, we introduce the new variables
$$x = A\tilde{x} = (\tilde{x}\cdot e, y):=(\eta,y)=(\eta,y_1,y_2), v = A\tilde{v}=(v_1,v_y), $$
where $\eta\in\mathbb{R}$, $y$ and $v_y\in\mathbb{R}^2$. Note that both $\eta$ and $v_y$ are independent of $y$, and $ \moc{\tilde{v}-\tilde{u}}=\moc{v-u} =\moc{v_1-u_1}$.

Define
$$ f(t,x,v) = \tilde{f}(t,A^{-1}x,A^{-1}v).$$
Then \eqref{Bolzmann Equation} is transformed into
        \begin{align}\label{Change variable BE}
        \left\{\begin{array}{rll}
          &\partial_t f + v_1\partial_{\eta}f + v_y\partial_{y}f = Q(f,f),\\[2mm]
          &f_0(x,v)=\tilde{f}_0( A^{-1}x,A^{-1}v),
          \end{array}\right.
        \end{align}
where we have also performed the change of variable $ \tilde \omega = A^{-1}\omega.$ The collision Boltzmann operator $Q(\cdot,\cdot)$ is given by
        \begin{align*}
          Q(f,f)&=\int_{\mathbb{R}^3}\int_{\S^2}\moc{\tilde{u}-\tilde{v}}^{\gamma}\mathbf{b} \left(\dfrac{A^{-1} v - A^{-1}u}{\moc{A^{-1} v - A^{-1} u }}\cdot A^{-1}\omega\right)\\
          &\quad\times\left( \tilde{f} (t, A^{-1}x, A^{-1}v')\tilde{f}(t,A^{-1}x,A^{-1}u') -\tilde{f}(t,A^{-1}x,A^{-1}v)\tilde{f}(t,A^{-1}x,A^{-1}u)\right) du d\omega\\
                   &=\int_{\mathbb{R}^3}\int_{\S^2}\moc{v-u}^{\gamma}\mathbf{b}(\cos\theta)
                   \left( f(t, x, v')f(t,x,u')-f(t,x,v)f(t,x,u)\right) du d\omega
                   \notag\\
         =&Q_{\textrm{\textrm{gain}}}(f,f) - Q_{\textrm{\textrm{loss}}}(f,f)
        \end{align*}
        with
        $$u'=A\tilde{u}' = u+\omega\cdot(v-u)\omega ,v' = A\tilde{v}' =v-\omega\cdot(v-u)\omega,$$ and $\theta=\dfrac{v-u}{\moc{v-u}}\cdot \omega.$

Our goal is to investigate the well-posedness of the Cauchy problem \eqref{Change variable BE} near vacuum within anisotropic low regularity Besov spaces. The vacuum regime presents a fundamental challenge in the analysis of the Boltzmann equation, as the collision operator becomes highly degenerate when the density vanishes, leading to a severe loss of coercivity and regularizing effects. By working in anisotropic critical Besov spaces adapted to the scaling of the equation, we capture the minimal regularity framework compatible with the vacuum state. Moreover, the anisotropic structure introduced through the change of variables reveals a directional decomposition aligned with the relative velocity, which is essential for compensating the degeneracy near vacuum. This allows us to identify and exploit a hidden structure in the equation, combining transport effects with the geometry of collisions.

       \subsection{Main results}
       Before stating our main results, we introduce several basic norms.
 For $l\in\R$, we define the time dependent weight
 \begin{align}\label{ww}
w(t,x,v)=\langle x - (t+1)v\rangle^{l}.
\end{align}
We then introduce the energy-type anisotropic Chemin-Lerner space
        \begin{equation}\label{Energy1}
          \mathscr{E}_T(f) = \norm{f(t,x,v)}_{\tilde{L}_T^{\infty}\tilde{L}_{v,\eta}^p(B^{2/p}_{p,1,w})},\ x=(\eta,y),
        \end{equation}
        for $t\in[0,T]$, where the norm $\tilde{L}_T^{\infty}\tilde{L}_{v,\eta}^p(B^{2/p}_{p,1,w})$ will be defined in Section \ref{sec-nfs}.
        Correspondingly, we define the initial energy by
        \begin{equation}\label{Energy1-T=0}
          \mathscr{E}_0(f) = \norm{f(0,x,v)}_{\tilde{L}_{v,\eta}^p(B^{2/p}_{p,1,w_0})}=\norm{f_0(x,v)}_{\tilde{L}_{v,\eta}^p(B^{2/p}_{p,1,w_0})},
        \end{equation}
where $w_0(x,v) = w(0,x,v) = \langle x-v\rangle^{l}.$

Next, we introduce the bilinear dissipation functional
\begin{equation}\label{D_T(f,g)}
          \mathscr{D}_T(f,g) = \sum_{j,k\geq -1} 2^{\frac{2(j+k)}{p}} \norm{\moc{v-u}^{\frac{1}{p}}\norm{w(t,x,v)\Delta_jf(t,x,v)}_{L_y^p}
          \norm{w(t,x,u)\Delta_kg(t,x,u)}_{L_y^p} }_{L^p_{t,v,u,\eta}},
        \end{equation}
 for $t\in[0,T]$, and in particular, we denote
        \begin{align}\label{Energy2}
          \mathscr{D}_T(f) =&\mathscr{D}_T(f,f)\notag\\=& \sum_{j,k\geq -1} 2^{\frac{2(j+k)}{p}} \norm{\moc{v-u}^{\frac{1}{p}}\norm{w(t,x,v)\Delta_jf(t,x,v)}_{L_y^p}
          \norm{w(t,x,u)\Delta_kf(t,x,u)}_{L_y^p} }_{L^p_{t,v,u,\eta}}.
        \end{align}

We are now ready to state our main result.
        \begin{theorem}\label{Well-posedeness}
          Suppose that $p> 1$ and $\frac{3}{p}-2 < \gamma <\frac{2}{p}-1$, with the endpoint case $p=1$ corresponding to $\gamma=1$. Let $l>\max\left\{ 1-\dfrac{1}{p},3-\dfrac{4}{p}+\gamma\right\}.$
          Assume there exists $\vps>0$ such that
          $$ \mathscr{E}_0(f)=\sum_{j\geq -1} 2^{\frac{2j}{p}}\norm{w_0(x,v)\Delta_j f_0(x,v)}_{L_{v,x}^{p}}\leq\varepsilon, $$
        where $w_0(x,v)=\langle x-v\rangle^{l}$.

        Then  there exists a unique global solution $f(t,x,v)$ to \eqref{Change variable BE} satisfying
        \begin{align}\label{eng-es}
   \mathscr{E}_T(f)+\mathscr{D}^{\frac{1}{2}}_T(f)\lesssim \mathscr{E}_0(f), ~~\forall T\geq0.
        \end{align}
        \end{theorem}
Two remarks concerning Theorem \ref{Well-posedeness} are given below.
\begin{remark}
In Theorem \ref{Well-posedeness}, since $p \in [1,+\infty)$, we have $\gamma \in (-2,1]$, which covers both moderately soft and hard potentials. The case of very soft potentials, corresponding to $\gamma \in (-3,-2]$, remains an open problem and requires further investigation.
\end{remark}
\begin{remark}
It is well known that the classical functional framework for establishing the well-posedness of the non-cutoff Boltzmann equation near vacuum is the weighted $L^\infty$ space; see, for instance, \cite{is-84,guo-vpb-01}. In Theorem \ref{Well-posedeness}, we construct a unique global strong solution to the Boltzmann equation near vacuum in the anisotropic Besov space $\tilde{L}_T^{\infty}\tilde{L}_{v,\eta}^p(B^{2/p}_{p,1,w})$ for $p\in[1,+\infty)$. Note that the embedding $\tilde{L}^p_\eta(B^{2/p}_{p,1}) \hookrightarrow L^\infty_{\eta,y}(\mathbb{R}^3)$ fails. In this sense, the solution obtained here can be regarded as a low-regularity global solution.
Moreover, we refer to $B_{p,1,w}^{2/p}$ as a critical Besov space, since the embedding $B_{p,1,w}^{2/p}\hookrightarrow L_{y}^{\infty,w}$ is critical.

\end{remark}

      \subsection{Related works}
Over the past several decades, the global well-posedness of the Cauchy problem for the Boltzmann equation has been a central topic in kinetic theory. Depending on the nature of the initial data, the existing theory can be broadly divided into two fundamentally different regimes.
The first regime concerns perturbations around Maxwellian equilibrium, where the initial data are assumed to be close to, or controlled by, a Maxwellian distribution. This setting has been extensively studied and is by now relatively well understood, with a variety of approaches including the spectral method \cite{NI,Sh,U74,U-s,UY-AA}, the energy method \cite{Guo-IUMJ,Guo-L,Liu-Yang-Yu,LY-S} and the comparison principle \cite{gpv-09}, etc.
In contrast, the second regime deals with non-Maxwellian solutions, where the initial data may be large, rough, or far from equilibrium. Typical examples include data without regularity \cite{DL}, perturbations around non-Maxwellian profiles such as polynomial states \cite{AEP-87,Car-1933,guo-liu-17}, and, in particular, the near-vacuum regime, a setting that has attracted increasing attention and will be the focus of the present paper.

The study of the Boltzmann equation near vacuum dates back to the classical work of Illner and Shinbrot \cite{is-84}, where weighted
$L^\infty$
 solutions were constructed for the hard-sphere model via a contraction mapping argument. This line of research was subsequently considered for related kinetic systems. In particular, Guo \cite{guo-vpb-01} established the well-posedness of the Vlasov-Poisson-Boltzmann (VPB) system near vacuum using the method of characteristics for moderately soft potentials $\ga\in(-2,0]$, and this result was later generalized to hard potentials by Duan-Yang-Zhu \cite{dyz-05,dyz-06}.
 Around the same time, Chae-Ha-Hwang and Duan-Zhang-Zhang \cite{chh-jde,dzz-06} proved $L^1$ stability of the VPB system near vacuum via a Lyapunov functional approach, which was later extended to
$L^p$
 stability by Ha-Yamazaki-Yun \cite{hyy-08}. Parallel to these developments, the small-data global theory for the Boltzmann equation itself has seen significant progress in various functional settings. Ars\'{e}nio \cite{ad-11} established global existence of mild solutions for small initial data in $L^D$ spaces. Later, Chen-Denlinger-Pavlovi\'{c} \cite{cdp-21} developed a bilinear spacetime estimate framework to obtain small-data global well-posedness. For cutoff kernels with soft potentials, He-Jiang \cite{hj-17,hj-23} studied well-posedness and scattering, as well as the Cauchy problem for small initial data, while further $L^p$ estimates on the gain term of the collision operator and their applications were obtained in \cite{hjkl-24}. For further related investigations, we refer the reader to the references therein in the above works.

More recently, extensions to other kinetic models have been studied, including the relativistic Boltzmann equation \cite{st-10,tc-18}, the quantum Boltzmann equation \cite{ow-22}, and the Vlasov-Yukawa-Boltzmann system \cite{ch-15,hxz-25}.
We also note that Luk \cite{jj-19} studied the well-posedness of the Landau equation near vacuum using the vector field method, and this result was later extended to the non-cutoff Boltzmann equation by Chaturvedi \cite{cs-21}.

Despite these advances, the well-posedness theory for the Boltzmann equation near vacuum remains far from complete, especially in critical function spaces. Very recently, Chen-Shen-Zhang \cite{csz-23,csz-24} studied both ill-posedness and well-posedness in critical Sobolev spaces via a dispersive approach. We also note that there has been some progress on the well-posedness of the Cauchy problem for the Boltzmann equation near Maxwellian equilibrium in low-regularity spaces; see, for instance, \cite{dlss-21,dlx-16,ds-18,ms-16,SS,ISV-26}.

In this paper, we develop a new approach to the Boltzmann equation near vacuum in low-regularity Besov spaces. More precisely, we establish the global well-posedness of strong solutions near vacuum in the critical anisotropic Besov space $\tilde{L}_T^{\infty}\tilde{L}_{v,\eta}^p(B^{2/p}_{p,1,w})$.


\subsection{Strategies and ideas}
This paper establishes the global well-posedness of the Boltzmann equation near vacuum in anisotropic low-regularity Besov spaces. The main contributions and methodologies can be summarized as follows.

\medskip
\noindent$\bullet$ \textit{Weighted anisotropic critical Besov spaces.}

To capture the energy-type norms of solutions to the Cauchy problem, we introduce weighted anisotropic critical Besov spaces
$\tilde{L}_T^{\infty}\tilde{L}_{v,\eta}^p(B^{2/p}_{p,1,w})$ $(p\in[1,+\infty))$, where the Besov space $B^{2/p}_{p,1,w}$ is critical in the sense that
$B_{p,1,w}^{2/p}\hookrightarrow  L_{y}^{\infty,w}$.
Moreover, the weight $w(t,x,v)=\langle x-(t+1)v\rangle^{l}$ with $l>\max\left\{ 1-\dfrac{1}{p},3-\dfrac{4}{p}+\gamma\right\}$
is carefully chosen to capture the transport effects and ensure integrability within the anisotropic framework. By working in these weighted critical spaces, we are able to simultaneously capture the low-regularity structure of the solution near vacuum and the directional propagation induced by free transport, which is essential for controlling nonlinear interactions in the absence of strong coercivity from the collision operator.

\noindent$\bullet$ \textit{Div-curl lemma and dissipation norm.}

By applying an orthogonal transformation, we convert the original Boltzmann equation into a form exhibiting a divergence-curl structure:
\begin{equation}
\partial_t f + v_1 \partial_\eta f + v_y \cdot \nabla_y f = Q(f,f).\notag
\end{equation}
Using this structure together with the div-curl lemma (Lemma \ref{div-curl Lemma}), we derive bilinear dissipation estimates of the form
\begin{align*}
       \int_0^T&\int_{\R^3} \int_{-\infty}^{+\infty}|v-u|\left(\int_{\mathbb{R}^2}\moc{ w(v)\Delta_j f(v)}^pdy\right)\left(\int_{\mathbb{R}^2}\moc{ w(u)\Delta_k g(u)}^{p}dy\right) d\eta dudvdt \\
               \lesssim&  \norm{w_0(v)\Delta_j f_0(v)}^p_{L_{v,x}^p}\norm{w_0(u)\Delta_k g_0(u)}^p_{L_{u,x}^p} + \norm{w(v)\Delta_j f(T,v)}^p_{L_{v,x}^p}\norm{w(u)\Delta_k g(T,u)}^p_{L_{u,x}^p}+h.o.t.
        \end{align*}
Accordingly, the bilinear dissipation norm is defined as
\begin{equation}
          \mathscr{D}_T(f,g) = \sum_{j,k\geq -1} 2^{\frac{2(j+k)}{p}} \norm{\moc{v-u}^{\frac{1}{p}}\norm{w(t,x,v)\Delta_jf(t,x,v)}_{L_y^p}
          \norm{w(t,x,u)\Delta_kg(t,x,u)}_{L_y^p} }_{L^p_{t,v.u,\eta}}.\notag
        \end{equation}
This functional not only captures the anisotropic interactions between dyadic components but also encodes the effect of relative velocities in the collision operator, which is crucial near vacuum where the standard coercivity estimates fail. As a result, the final energy estimates take the form
\begin{align*}
\mathscr{E}_T^2(f) + \mathscr{D}_T(f)
\lesssim \mathscr{E}_0^2(f)
+ \mathscr{E}_T^{2-\frac{2}{p}}(f)\mathscr{D}_T^{\frac{2}{p}}(f)
+ \mathscr{E}_T^{2-\frac{1}{p}}(f)\mathscr{D}_T^{\frac{1}{p}}(f).
\end{align*}

\noindent$\bullet$ \textit{Bilinear dissipation control of the nonlinear collision operator.}

We establish nonlinear estimates for the Boltzmann collision operator in a Besov-type framework by introducing a bilinear dissipation mechanism that directly controls the nonlinear term. Using Littlewood-Paley's decomposition and Bony's paraproduct, each dyadic component of
$Q(f,g)$ is carefully estimated, allowing the bilinear interactions to be summed and absorbed into the dissipation functional
$\mathscr{D}_T(f,g)$. This yields a unified bound valid for all
$p\in[1,+\infty)$.

\subsection{Structure of the paper}
The rest of the paper is organized as follows: Section \ref{sec-nfs} introduces the functional framework, including weighted anisotropic critical Besov spaces and Littlewood-Paley decomposition. In Section \ref{sec-nones}, we derive key nonlinear estimates, in particular a bilinear bound for the collision operator via the bilinear dissipation functional. Section \ref{ape-sec} employs a div-curl lemma to establish the {\it a priori} energy estimates. Section \ref{sec-glo} combines these results to prove the global existence and uniqueness of strong solutions for small initial data. Finally, Section \ref{sec-app} collects technical lemmas, including the div-curl lemma and weight estimates.

        \section{Notations and Function Spaces}\label{sec-nfs}
        Throughout the paper, $C$ denote some generic positive (generally large) constant and $i,j,k$ denote the integer, $p', p\geq1$ and satisfy $ \dfrac{1}{p}+\dfrac{1}{p'} = 1.$
        For two quantities $A$ and $B$, $A\lesssim B$ means that there is a generic constant $C>0$ such that $A\leq CB$. We denote $\langle x \rangle:=\sqrt{1+\moc{x}^2}$ for any $x\in \mathbb{R}^3$.

       Since the key nonlinear estimates rely on a dyadic decomposition in the Fourier variable, we briefly recall the Littlewood-Paley theory and the associated function spaces, including Besov and Chemin-Lerner spaces. For more details, we refer the reader to \cite{BCD}.

We begin with the Fourier transform. Throughout this paper, the Fourier transform is taken only with respect to the variable $y$, while $\eta$, $v$, and $t$ are treated as parameters. For fixed $(t,\eta,v)$, the Fourier transform of a Schwartz function $f(t,\eta,y,v)\in\mathcal{S}(\mathbb{R}^2_y)$ is defined by
        $$ \hat{f}(t,\eta,z,v):=\int_{\mathbb{R}^2} e^{-iy\cdot z}f(t,\eta,y,v)dy.$$

        Next, we introduce a dyadic partition of unity in $\mathbb{R}^2$. Let $(\varphi,\chi)$ be smooth functions taking values in $[0,1]$ such that $\varphi$ is supported in the annulus $\mathbb{C}\left(0,\dfrac{3}{4},\dfrac{8}{3}\right)=\left\{ z\in \mathbb{R}^2: \dfrac{3}{4}\leq \moc{z}\leq \dfrac{8}{3}\right\}$ and $\chi$ is supported in the ball $\mathbb{B}\left( 0,\dfrac{4}{3} \right) = \left\{ z\in\mathbb{R}^2:\moc{z}\leq \dfrac{4}{3} \right\},$ with
        $$ \chi(z)+\sum_{j\geq 0} \varphi(2^{-j}z)=1, \quad\quad\forall z\in \mathbb{R}^2. $$
        $$ \sum_{j\in \mathbb{Z}}\varphi(2^{-j}z) =1 \quad\quad \forall z\in\mathbb{R}^{2}- \{0\}.$$ The nonhomogeneous dyadic blocks of $f$ are defined as follows:
        $$\Delta_{-1}f := \chi(D)f = \tilde{\psi} * f = \int_{\mathbb{R}^2} \tilde{\psi}(y) f(z-y) \, dy, \quad \text{with } \tilde{\psi} = \mathcal{F}^{-1}\chi;$$
        $$\Delta_j f := \varphi(2^{-j}D)f = 2^{2j} \int_{\mathbb{R}^2} \psi(2^j y) f(z-y) \, dy,\quad \text{with } \psi = \mathcal{F}^{-1}\varphi, \quad j \geq 0,$$
        where $*$ is the convolution operator with respect to the variable $z$ and $\mathcal{F}^{-1}$ denotes
        the inverse Fourier transform. Define the low frequency cut-off operator $S_j$ ($j \geq -1$) by
        $$S_j f := \sum_{k \leq j-1} \Delta_k f,$$
        with the convention that $S_0 f = \Delta_{-1} f$ and $S_{-1} f = 0$. One also sees that
        $$ S_{j}f = \chi({2^{-j}D}):= \int_{\mathbb{R}^2} \tilde{\psi}(2^j z) f(y - z)dz,\quad\quad \forall j \geq -1. $$

        With these notions, the nonhomogeneous Littlewood-Paley decomposition of  $f$ is given by
        $$f = \sum_{j \geq -1} \Delta_j f.$$ Having defined the linear operators $\Delta_j$ for $j\geq -1,$ we give the definition of nonhomogeneous anisotropic Besov space as follows.

      \begin{definition}
Let $1 \le p,q,r \le \infty$ and $s \in \mathbb{R}$.

The nonhomogeneous Besov space is defined by
\[
B_{p,r}^s := \Big\{ g \in \mathcal{S}'(\mathbb{R}^2_y) :
g = \sum_{j \ge -1} \Delta_j g,\ \|g\|_{B_{p,r}^s} < \infty \Big\},
\]
where $\mathcal{S}'(\mathbb{R}^2_y)$ denotes the space of Schwartz distributions on $\mathbb{R}^2_y$.

The nonhomogeneous anisotropic Besov space is defined by
\[
\tilde{L}_{v,\eta}^p(B_{p,r}^s)
:= \Big\{ f \in \mathcal{S}'(\mathbb{R}^3_v \times \mathbb{R}_\eta \times \mathbb{R}^2_y) :
f = \sum_{j \ge -1} \Delta_j f,\
\|f\|_{\tilde{L}_{v,\eta}^p(B_{p,r}^s)} < \infty \Big\}.
\]

We also define the time-dependent Chemin--Lerner type space
\[
\tilde{L}_T^q \tilde{L}_{v,\eta}^p(B_{p,r}^s)
:= \Big\{ f \in \mathcal{S}'(\mathbb{R}_t \times \mathbb{R}^3_v \times \mathbb{R}_\eta \times \mathbb{R}^2_y) :
f = \sum_{j \ge -1} \Delta_j f,\
\|f\|_{\tilde{L}_T^q \tilde{L}_{v,\eta}^p(B_{p,r}^s)} < \infty \Big\}.
\]

The corresponding norms are given by
\[
\|g\|_{B_{p,r}^s}
= \left( \sum_{j \ge -1} \big( 2^{js} \|\Delta_j g\|_{L_y^p} \big)^r \right)^{\frac{1}{r}},
\]
\[
\|f\|_{\tilde{L}_{v,\eta}^p(B_{p,r}^s)}
= \left( \sum_{j \ge -1} \big( 2^{js} \|\Delta_j f\|_{L_{v,\eta,y}^p} \big)^r \right)^{\frac{1}{r}},
\]
and
\[
\|f\|_{\tilde{L}_T^q \tilde{L}_{v,\eta}^p(B_{p,r}^s)}
= \left( \sum_{j \ge -1} \big( 2^{js} \|\Delta_j f\|_{L_t^q L_{v,\eta,y}^p} \big)^r \right)^{\frac{1}{r}}, \ t\in[0,T],
\]
with the usual modification when $ p, q, r=\infty.$

\medskip

Moreover, for the time-dependent weight
\[
w(t,x,v) = \langle x-(t+1)v \rangle^l,
\]
we define the weighted anisotropic Besov norm by
\[
\|f\|_{\tilde{L}_T^q \tilde{L}_{v,\eta}^p(B_{p,r,w}^s)}
= \left( \sum_{j \ge -1} \big( 2^{js} \| w(t,x,v)\Delta_j f \|_{L_t^q L_{v,\eta,y}^p} \big)^r \right)^{\frac{1}{r}},\ t\in[0,T],
\]
again with the usual modification for $ p, q, r=\infty.$ Throughout the paper, we use $L^{p,w}$ to denote the weighted $L^p$ space.

\medskip

In particular, the case $r=1$ and $s=\frac{2}{p}$ (with $1 \le p < \infty$) plays a central role in this paper. We will mainly work in the spaces
\[
\tilde{L}_{v,\eta}^p(B_{p,1,w_0}^{2/p})
\quad \text{and} \quad
\tilde{L}_T^\infty \tilde{L}_{v,\eta}^p(B_{p,1,w}^{2/p}).
\]
\end{definition}

Note that for fixed $(t,v,\eta)$, the following weighted embedding holds: $$ B_{p,1,w}^{2/p}\hookrightarrow L_{y}^{\infty,w},$$
where, in the case that $w$ only depends on $y$, the result is proved in \cite[pp.~156]{et-96}. More precisely, we have
\begin{equation}\label{Embedding inequailty}
          \norm{w(v)f}_{L_y^{\infty}}\leq C \sum_{j\geq
        -1} 2^{\frac{2j}{p}}\norm{w(v)\Delta_j f}_{L_{y}^p}.
        \end{equation}
        and a detailed proof is provided in Appendix \ref{sec-app}.


        \section{Nonlinear Estimates}\label{sec-nones}
In this section, we establish the key estimates for the nonlinear Boltzmann collision operator $Q(f,f)$. The main idea is to exploit the bilinear dissipation functional \eqref{D_T(f,g)} to control the bilinear operator $Q(f,g)$. This approach allows us to effectively capture the nonlinear interactions through a unified dissipative structure.
The resulting estimates are valid for $p\in(1,+\infty)$ with $\gamma\in\left(\frac{3}{p}-2,\frac{2}{p}-1\right)$, as well as for the endpoint case $p=1$ and $\gamma=1$.

        \begin{proposition}\label{prop-Nonhomogeneous term}
          For $p\in (1,+\infty)$ with $\gamma\in \left(\dfrac{3}{p}-2,\dfrac{2}{p} -1\right)$ or $p=1$ with $\gamma = 1$, then for any $f,g$, and $T\geq0$, it holds that
          $$ \int_{0}^{T}\sum_{j\geq-1}2^{\frac{2j}{p}} \norm{w(v)\Delta_j Q(f,g)}_{L_{v,x}^p}dt \lesssim \mathscr{D}_T(f,g), $$
          where $\mathscr{D}_T(f,g)$ is defined by \eqref{Energy2}. In particular,
          \begin{equation}\label{Estimate of Q(f,f)}
            \int_{0}^{T}\sum_{j\geq-1}2^{\frac{2j}{p}} \norm{w(v)\Delta_j Q(f,f)}_{L_{v,x}^p}dt \lesssim \mathscr{D}_T(f).
          \end{equation}
        \end{proposition}

        \begin{proof}
    We write for $p\in[1,+\infty)$
          \begin{align}
            \int_{0}^{T}&\sum_{j\geq-1}2^{\frac{2j}{p}} \norm{w(v)\Delta_j Q(f,g)}_{L_{v,x}^p}dt\notag\\
             &\leq  \int_{0}^{T}\sum_{j\geq-1}2^{\frac{2j}{p}} \norm{w(v)\Delta_j Q_{\textrm{gain}}(f,g)}_{L_{v,x}^p}dt+\int_{0}^{T}\sum_{j\geq-1}2^{\frac{2j}{p}} \norm{w(v)\Delta_j Q_{\textrm{loss}}(f,g)}_{L_{v,x}^p}dt\notag\\
             &:= I_1 + I_2.\label{ni1-2}
          \end{align}
          The computation is divided into following two cases.

          \noindent\underline{{\it Case 1: $p\in (1,+\infty)$ with $\gamma\in \left(\dfrac{3}{p}-2,\dfrac{2}{p} -1\right)$.}} We estimate $I_1$ and $I_2$ in \eqref{ni1-2} individually.
For $I_1$, applying Bony's decomposition, we obtain
          $$fg=\sum_{i\geq -1} S_{i-1} f\Delta_ig +\sum_{i\geq -1} S_{i-1} g\Delta_if + \sum_{\moc{i-i'}\leq 1} \Delta_i f\Delta_{i'}g,$$ where $S_{i}f=\sum_{-1\leq k \leq i-1} \Delta_k f$. Using H$\ddot{\text{o}}$lder's inequality, Minkowski's inequality, and Lemma \ref{lem-gain}, it follows
          \begin{align*}
             I_1&= \int_{0}^{T}\sum_{j\geq-1}2^{\frac{2j}{p}} \norm{w(v)\Delta_j Q_{\textrm{gain}}(f,g)}_{L_{v,x}^p}dt\\
                &\lesssim \sum_{j\geq-1}2^{\frac{2j}{p}}\int_{\mathbb{S}^2} \moc{\mathbf{b}(\cos \theta)}d\omega\\
                &\qquad\times\int_{0}^{T} \norm{\int_{\mathbb{R}^3}w(v)w^{-1}(v')w^{-1}(u') \moc{v-u}^{\gamma}w(v')w(u')\Delta_j(f(v')g(u'))  du}_{L_{v,x}^p}dt\\
                &\lesssim \sum_{j\geq-1}2^{\frac{2j}{p}}\int_{\mathbb{S}^2} \moc{\mathbf{b}(\cos \theta)}d\omega\int_{0}^{T} \norm{
                \norm{ \moc{v-u}^{\frac{1}{p}}w(v')w(u')\Delta_j(f(v')g(u'))}_{L_{u}^p}I_{\textrm{gain}}
                }_{L_{v,x}^p}dt\\
                &\lesssim  \sum_{j\geq-1}2^{\frac{2j}{p}}\int_{\mathbb{S}^2} \moc{\mathbf{b}(\cos \theta)}d\omega\\
                &\qquad\times\int_{0}^{T}(t+1)^{-\gamma  + \frac{1}{p} - \frac{3}{p'}} \norm{
                \norm{ \moc{v-u}^{\frac{1}{p}}w(v')w(u')\Delta_j(f(v')g(u'))}_{L_{u}^p}
                }_{L_{v,x}^p}dt\\
                &\lesssim \sum_{j\geq-1}2^{\frac{2j}{p}}\int_{\mathbb{S}^2} \moc{\mathbf{b}(\cos \theta)}d\omega\norm{
                 \moc{v-u}^{\frac{1}{p}}w(v')w(u')\Delta_j(f(v')g(u'))
                }_{L_{t,u,v,x}^p}\\
                &\lesssim \sum_{j\geq-1}2^{\frac{2j}{p}}\norm{
                 \moc{v-u}^{\frac{1}{p}}w(v)w(u)\Delta_j(f(v)g(u))
                }_{L_{t,u,v,x}^p}\\
                &\lesssim \sum_{j\geq-1}2^{\frac{2j}{p}}\norm{
                 \moc{v-u}^{\frac{1}{p}}w(v)w(u)\Delta_j\left(\sum_{i\geq -1} S_{i-1} f(v)\Delta_ig(u)\right)
                }_{L_{t,u,v,x}^p}\\
                &+\sum_{j\geq-1}2^{\frac{2j}{p}}\norm{
                 \moc{v-u}^{\frac{1}{p}}w(v)w(u)\Delta_j\left(\sum_{i\geq -1} S_{i-1} g(u)\Delta_if(v)\right)
                }_{L_{t,u,v,x}^p}\\
                &+\sum_{j\geq-1}2^{\frac{2j}{p}}\norm{
                 \moc{v-u}^{\frac{1}{p}}w(v)w(u)\Delta_j\left(\sum_{\moc{i-i'}\leq 1} \Delta_i f(v) \Delta_{i'}g(u)\right)
                }_{L_{t,u,v,x}^p}\\
                &:= I_{1,1}+I_{1,2}+I_{1,3},
          \end{align*}
         where
\[
I_{\mathrm{gain}}
= \big\| |v-u|^{-\frac{1}{p}+\gamma}\, w(v)\, w^{-1}(v')\, w^{-1}(u') \big\|_{L_u^{p'}},
\]
whose estimate is provided by Lemma \ref{lem-gain}. We now estimate the terms $I_{1,1}$, $I_{1,2}$, and $I_{1,3}$ separately.

For $I_{1,1}$, due to the definition of  Littlewood-Paley's decomposition and inequality \eqref{bd - w(v)Delta_k}, we have
          \begin{align*}
             I_{1,1}&\lesssim \sum_{j\geq -1} \sum_{\moc{i-j}\leq 4}2^{\frac{2j}{p}}\norm{\moc{v-u}^{\frac{1}{p}}w(v)S_{i-1}f(v) w(u)\Delta_ig(u)}_{L_{t,u,v,x}^p}\\
                   &\lesssim \sum_{j\geq -1} \sum_{\moc{i-j}\leq 4}2^{\frac{2j}{p}}\norm{\moc{v-u}^{\frac{1}{p}}\norm{w(v)S_{i-1}f(v)}_{L_y^{\infty}} \norm{w(u)\Delta_ig(u)}_{L_y^p}}_{L_{t,u,v,\eta}^p},
          \end{align*}
          By the embedding inequality \eqref{Embedding inequailty},
          \begin{align*}
            \norm{w(v)S_{i-1}f}_{L^{\infty}_{y}}&\lesssim \sum_{k\geq -1}2^{\frac{2k}{p}} \norm{w(v)\Delta_k S_{i-1}f}_{L^{p}_y}=\sum_{k\geq -1}2^{\frac{2k}{p}} \norm{w(v)S_{i-1}\Delta_k f}_{L^{p}_y} \\
            &\lesssim \sum_{k\geq -1}2^{\frac{2k}{p}} \norm{w(v)\Delta_k f}_{L^{p}_y},
          \end{align*}
          thus, in view of Lemma \ref{Sum operator}, we arrive at
          \begin{align*}
             I_{1,1} &\lesssim \sum_{j,k\geq -1}\sum_{\moc{i-j}\leq 4}2^{\frac{2(j+k)}{p}} \norm{\moc{v-u}^{\frac{1}{p}}\norm{w(v)\Delta_k f(v)}_{L_y^{p}} \norm{w(u)\Delta_ig(u)}_{L_y^p}}_{L_{t,u,v,\eta}^p}\lesssim \mathscr{D}_T(f,g).
          \end{align*}
For $I_{1,2}$, similarly, we get from \eqref{Embedding inequailty}, \eqref{bd - w(v)Delta_k} and Lemma \ref{Sum operator} that
          \begin{align*}
           I_{1,2}&\leq \sum_{j\geq-1}2^{\frac{2j}{p}}\norm{\moc{v-u}^{\frac{1}{p}}w(v)w(u)\Delta_j\left(\sum_{i\geq -1} S_{i-1}g(u)\Delta_if(v)\right)}_{L_{t,u,v,x}^p}\\
                 &\lesssim \sum_{j\geq -1}\sum_{\moc{i-j}\leq 4}2^{\frac{2j}{p}}\norm{\moc{v-u}^{\frac{1}{p}}w(u)S_{i-1}g(u)w(v)\Delta_i f(v)}_{L_{t,u,v,x}^p}\\
                 &\lesssim \sum_{j\geq -1}\sum_{\moc{i-j}\leq 4}2^{\frac{2j}{p}}\norm{\moc{v-u}^{\frac{1}{p}}\norm{w(v)S_{i-1}g(u)}_{L_{y}^{\infty}}\norm{w(v)\Delta_i f(v)}_{L_{y}^{p}}}_{L_{t,u,v,\eta}^p}\\
                 &\lesssim \sum_{j\geq -1}\sum_{\moc{i-j}\leq 4}2^{\frac{2j}{p}}\norm{\moc{v-u}^{\frac{1}{p}}\sum_{k\geq -1} 2^{\frac{2k}{p}}\norm{w(u)\Delta_k g(u)}_{L_{y}^{p}}\norm{w(v)\Delta_i f(v)}_{L_{y}^{p}}}_{L_{t,u,v,\eta}^p}\\
                 &\lesssim \sum_{j,k\geq -1}\sum_{\moc{i-j}\leq 4}2^{\frac{2(j+k)}{p}}\norm{\moc{v-u}^{\frac{1}{p}} \norm{w(u)\Delta_k g(u)}_{L_{y}^{p}}\norm{w(v)\Delta_i f(v)}_{L_{y}^{p}}}_{L_{t,u,v,\eta}^p}\\
                 &\lesssim \mathscr{D}_T(f,g).\\
          \end{align*}
As for $I_{1,3}$, employing Lemma \ref{Sum operator} again, one has
          \begin{align*}
            I_{1,3}&=\sum_{j\geq-1}2^{\frac{2j}{p}}\norm{
                 \moc{v-u}^{\frac{1}{p}}w(v)w(u)\Delta_j\left(\sum_{\moc{i-i'}\leq 1} \Delta_i f(v)\Delta_{i'}g(u)\right)
                }_{L_{t,u,v,x}^p}\\
                  &\lesssim \sum_{j\geq -1}\sum_{\max\{i,i'\}\geq j-2}\sum_{\moc{i-i'}\leq1} 2^{\frac{2j}{p}}\norm{
                 \moc{v-u}^{\frac{1}{p}}w(v)\Delta_i f(v)w(u)\Delta_{i'}g(u)
                }_{L_{t,u,v,x}^p}\\
                  &\lesssim \sum_{j\geq -1}\sum_{\max\{i,i'\}\geq j-2}\sum_{\moc{i-i'}\leq1} 2^{\frac{2j}{p}}\norm{
                 \moc{v-u}^{\frac{1}{p}}\norm{ w(v)\Delta_i f(v)}_{L^p_y}\norm{w(u)\Delta_{i'}g(u)}_{L_y^{\infty}}
                }_{L_{t,u,v,\eta}^p}\\
                  &\lesssim \sum_{j\geq -1}\sum_{i \geq j-3} 2^{\frac{2j}{p}}\norm{
                 \moc{v-u}^{\frac{1}{p}}\norm{w(v)\Delta_i f(v)}_{L^p_y}\norm{w(u)g(u)}_{L_y^{\infty}}
                }_{L_{t,u,v,\eta}^p}\\
                  &\lesssim \sum_{j\geq -1}\sum_{i \geq j-3} 2^{\frac{2j}{p}}\norm{
                 \moc{v-u}^{\frac{1}{p}}\norm{ w(v)\Delta_i f(v)}_{L^p_y}\sum_{k\geq -1}2^{\frac{2k}{p}}\norm{w(u)\Delta_k g(u)}_{L_y^{p}}
                }_{L_{t,u,v,\eta}^p}\\
                  &\lesssim \sum_{j,k\geq -1}\sum_{i \geq j-3} 2^{\frac{2(j+k)}{p}}\norm{
                 \moc{v-u}^{\frac{1}{p}}\norm{ w(v)\Delta_i f(v)}_{L^p_y}\norm{w(u)\Delta_k g(u)}_{L_y^{p}}
                }_{L_{t,u,v,\eta}^p}\\
                  &\lesssim \mathscr{D}_T(f,g).
          \end{align*}
 Consequently, it follows
\begin{align}\label{i1-bd}
 I_1 \lesssim I_{1,1} + I_{1,2} + I_{1,3} \lesssim \mathscr{D}_T(f,g).
\end{align}
We now turn to compte $I_2$ in this case. We first get from Lemma \ref{lem-loss} that
          \begin{align*}
            I_2&=\int_{0}^{T}\sum_{j\geq-1}2^{\frac{2j}{p}} \norm{w(v)\Delta_j Q_{\textrm{loss}}(f,g)}_{L_{v,x}^p}dt\\
               &\lesssim \sum_{j\geq-1} 2^{\frac{2j}{p}}\int_{0}^{T}\norm{w(v)\int_{\mathbb{R}^3\times \mathbb{S}^2}\moc{v-u}^{\gamma}\mathbf{b}(\cos \theta)\Delta_j(f(v)g(u)) dud\omega}_{L_{v,x}^p}dt\\
               &\lesssim \sum_{j\geq-1} 2^{\frac{2j}{p}}\norm{\mathbf{b}(\cos\theta)}_{L^1(\mathbb{S}^2)}\notag\\
               &\qquad\times\int_{0}^{T}\norm{\int_{\mathbb{R}^3}\moc{v-u}^{\frac{1}{p}}w(v)w(u)\Delta_j(f(v)g(u)) \moc{v-u}^{-\frac{1}{p}+\gamma}w^{-1}(u)dw}_{L_{v,x}^{p}}dt\\
               &\lesssim \sum_{j\geq -1} 2^{\frac{2j}{p}}\int_{0}^{T}
               \norm{\norm{\moc{v-u}^{\frac{1}{p}}w(v)w(u)\Delta_j(f(v)g(u))}_{ L_{u}^p}\norm{\moc{v-u}^{-\frac{1}{p}+\gamma}w^{-1}(u)}_{L_{u}^{p'}}}_{L_{v,x}^p}dt\\
               &\lesssim \sum_{j\geq -1} 2^{\frac{2j}{p}}\int_{0}^{T}\norm{\norm{\moc{v-u}^{\frac{1}{p}}w(v)w(u)\Delta_j(f(v)g(u))  }_{L^p_u}I_{\textrm{loss}}}_{L_{v,x}^p}dt\\
               &\lesssim \sum_{j\geq -1} 2^{\frac{2j}{p}}\int_{0}^{T}(t+1)^{-\gamma + \frac{1}{p} - \frac{3}{p'}}\norm{\moc{v-u}^{\frac{1}{p}}w(v)w(u)\Delta_j(f(v)g(w))}_{L_{u,v,x}^p}dt\\
               &\lesssim \sum_{j\geq -1} 2^{\frac{2j}{p}}\norm{\moc{v-u}^{\frac{1}{p}}w(v)w(u)\Delta_j(f(v)g(u))}_{L_{t,u,v,x}^p}.
          \end{align*}
          Then performing the similar calculation as for obtaining the estimation for \eqref{i1-bd}, we see that $I_2$ is also bounded by $C\mathscr{D}_T(f,g)$ for $C>0$.

\noindent\underline{{\it Case 2: $p=1$ with $\gamma=1$.}}
        The argument in this case is similar to that of {\it Case 1}. For the gain term $I_1$, we first note that
          $$w(v)w^{-1}(v')w^{-1}(u')=\left(\dfrac{\langle x-(t+1)v\rangle}{\langle x-(t+1)v'\rangle\langle x-(t+1)u'\rangle}\right)^{l}\lesssim 1,
           $$
           hence
          \begin{align*}
            I_1&=\int_{0}^{T} \sum_{j\geq-1} 2^{2j}\norm{w(v)\Delta_jQ_{\textrm{gain}}(f,g)}_{L_{v,x}^1}dt\\
               &\lesssim \sum_{j\geq-1}2^{2j}\int_{\mathbb{S}^2} \moc{\mathbf{b}(\cos \theta)}d\omega\\
                &\quad \quad\times \int_{0}^{T} \norm{
                \norm{ \moc{v-u}w(v')w(u')\Delta_j(f(v')g(u'))}_{L_{u}^1}\norm{w(v)w^{-1}(v')w^{-1}(u')}_{L_u^{\infty}}
                }_{L_{v,x}^1}dt\\
                &\lesssim \sum_{j\geq-1}2^{2j}\int_{\mathbb{S}^2} \moc{b(\cos \theta)}d\omega
                \norm{ \moc{v-u}w(v')w(u')\Delta_j(f(v')g(u'))}_{L_{t,u,v,x}^1}\\
                &\lesssim \sum_{j\geq-1}2^{2j}\norm{ \moc{v-u}w(v)w(u)\Delta_j(f(v)g(u))}_{L_{t,u,v,x}^1}.
          \end{align*}
Similarly, for the loss term $I_2$, we obtain
          \begin{align*}
            I_2&=\int_{0}^{T} \sum_{j\geq-1} 2^{2j}\norm{w(v)\Delta_jQ_{\textrm{loss}}(f,g)}_{L_{v,x}^1}dt\\
               &\lesssim \sum_{j\geq-1}2^{2j} \int_{0}^{T} \norm{
                \norm{ \moc{v-u}w(v)w(u)\Delta_j(f(v)g(u))}_{L_u^1}\norm{w^{-1}(u)}_{L_u^{\infty}}
                }_{L_{v,x}^1}dt\\
                &\lesssim \sum_{j\geq-1}2^{2j}
                \norm{\moc{v-u}w(v)w(u)\Delta_j(f(v)g(u))}_{L_{t,u,v,x}^1}.
          \end{align*}
Consequently,
          \begin{equation}
            I_1 + I_2 \lesssim \sum_{j\geq-1}2^{2j}\norm{\moc{v-u}w(v)w(u)\Delta_j(f(v)g(u))}_{L_{t,u,v,x}^1}.\notag
          \end{equation}
Arguing as in the estimate leading to \eqref{i1-bd}, we deduce that
          $$ \sum_{j\geq-1}2^{2j}\norm{\moc{v-u}w(v)w(u)\Delta_j(f(v)g(u))}_{L_{t,u,v,x}^1}\lesssim \mathscr{D}_T(f,g). $$
 Therefore,
          $$ \int_{0}^{T}\sum_{j\geq-1}2^{2j} \norm{w(v)\Delta_j Q(f,g)}_{L_{v,x}^1}dt \lesssim \mathscr{D}_T(f,g).$$
This ends the proof of Proposition \ref{prop-Nonhomogeneous term}.
        \end{proof}

            \section{Bilinear dissipation and the {\it a priori} estimate}\label{ape-sec}
  In this section, we establish key bilinear dissipation estimates and derive global {\it a priori} energy-type bounds for solutions of the rescaled Boltzmann equation \eqref{Change variable BE}. Building on the functional framework introduced earlier, we first obtain estimates for the bilinear dissipation functional
$\mathscr{D}_T(f,g)$ in terms of the initial and temporal energy norms. These estimates then allow us to control the nonlinear term in the equation and, under a suitable smallness assumption, to close the {\it a priori} energy estimates. 

       \begin{proposition}\label{Estimation of D_T(f,g) prop}
Suppose that $f$ and $g$ are global strong solutions to \eqref{Change variable BE} with initial data $f_0$ and $g_0$, respectively. Let $p\in (1,+\infty)$ with $\gamma\in \left(\dfrac{3}{p}-2,\dfrac{2}{p} -1\right)$ or $p=1$ with $\gamma = 1$. Then, for any $T\geq0$, it holds that
\begin{equation}\label{Estimation of D_T(f,g)}
\mathscr{D}_T(f,g) \lesssim \mathscr{E}_0(f)\mathscr{E}_0(g)
+ \mathscr{E}_T(f)\mathscr{E}_T(g)
+ \mathscr{E}_T(f)\mathscr{E}_T^{1-\frac{1}{p}}(g)\mathscr{D}_T^{\frac{1}{p}}(g)
+ \mathscr{E}_T(g)\mathscr{E}_T^{1-\frac{1}{p}}(f)\mathscr{D}_T^{\frac{1}{p}}(f).
\end{equation}
In particular, for any $T\ge 0$, we have
\begin{equation}\label{Estimation of D_T(f,f)}
\mathscr{D}_T(f) \lesssim \mathscr{E}_0^2(f)
+ \mathscr{E}_T^2(f)
+ \mathscr{E}_T^{2-\frac{1}{p}}(f)\mathscr{D}_T^{\frac{1}{p}}(f).
\end{equation}
\end{proposition}

        \begin{proof} It suffices to prove \eqref{Estimation of D_T(f,g)}.
Recalling \eqref{ww}, we see that $f$ and $g$ satisfy
\begin{align}\label{f-eq}
(\partial_t + v\cdot \nabla_x)\big(w(v)\Delta_j f\big)
= w(v)\Delta_j Q(f,f),
\end{align}
with initial data $f(0,x,v)=f_0(x,v)$, and
\begin{align}\label{g-eq}
(\partial_t + u\cdot \nabla_x)\big(w(u)\Delta_k g\big)
= w(u)\Delta_k Q(g,g),
\end{align}
with $g(0,x,u)=g_0(x,u)$.

We multiply \eqref{f-eq} and \eqref{g-eq} by
\[
p\,|w(v)\Delta_j f|^{p-2} w(v)\Delta_j f
\quad \text{and} \quad
p\,|w(u) \Delta_k g|^{p-2} w(u) \Delta_kg,
\]
respectively. Integrating the resulting equations over the variable $y\in\mathbb{R}^2$, we obtain
          \begin{equation}
            \left\{
            \begin{array}{ll}
             \partial_t \int_{\mathbb{R}^2} \moc{w(v)\Delta_j f(v)}^{p}dy + &v_1 \cdot \partial_{\eta}\int_{\mathbb{R}^2}\moc{ w(v)\Delta_j f(v)}^pdy\\[2ex]& = \int_{\mathbb{R}^2}p\moc{w(v)\Delta_j f(v)}^{p-2}w(v)\Delta_j f(v) w(v)\Delta_j Q(f,f)dy,\\[2ex]
             \partial_t \int_{\mathbb{R}^2}\moc{ w(u)\Delta_k g(u)}^{p}dy + &u_1 \cdot \partial_{\eta}\int_{\mathbb{R}^2}\moc{ w(u)\Delta_k g(u)}^pdy\\[2ex]&= \int_{\mathbb{R}^2}p\moc{w(u)\Delta_k g(u)}^{p-2}w(u)\Delta_k g(u) w(u)\Delta_k Q(g,g)dy .
            \end{array}
            \right.\notag
          \end{equation}
 Then div-curl Lemma \ref{div-curl Lemma} yields
          \begin{align*}
       & \int_0^T \int_{-\infty}^{+\infty}(v_1-u_1)\left(\int_{\mathbb{R}^2}\moc{ w(v)\Delta_j f(v)}^pdy\right)\left(\int_{\mathbb{R}^2}\moc{ w(u)\Delta_k g(u)}^{p}dy\right) d\eta dt \\
                &\quad\quad\leq \left| \int_{\eta < \eta'} \int_{\mathbb{R}^2} \moc{w_0(v)\Delta_j f_0(\eta,v)}^{p}dy \int_{\mathbb{R}^2}\moc{ w_0(u)\Delta_k g_0(\eta',u)}^{p}dy  d\eta d\eta' \right| \\
                &\quad\quad+\left| \int_{\eta < \eta'} \int_{\mathbb{R}^2} \moc{w(v)\Delta_j f(T,\eta,v)}^{p}dy \int_{\mathbb{R}^2}\moc{ w(u)\Delta_k g(T,\eta',u)}^{p}dy  d\eta d\eta' \right| \\
                &\quad\quad+\left| \int_0^T \int_{-\infty}^{+\infty} \left( \int_{-\infty}^{\eta} \int_{\mathbb{R}^2} \moc{w(v)\Delta_j f(t,\eta',v)}^{p}dy d\eta' \right)\right.\\
                &\quad\quad\quad\quad\quad\quad\quad\quad\times \left. \left(\int_{\mathbb{R}^2}p\moc{w(u)\Delta_k g(u)}^{p-2}w(u)\Delta_k g(t,\eta,u) w(u)\Delta_k Q(g,g)dy\right)d\eta dt \right| \\
                &\quad\quad+\left| \int_0^T \int_{-\infty}^{+\infty} \left( \int_{\eta}^{+\infty}\int_{\mathbb{R}^2}\moc{ w(u)\Delta_k g(t,\eta',u)}^{p}dy d\eta' \right)
       \right.\\&\quad\quad\quad\quad\quad\quad\quad\quad\times \left. \left(\int_{\mathbb{R}^2}p\moc{w(v)\Delta_j f(t,\eta,v)}^{p-2}w(v)\Delta_j f(v) w(v)\Delta_j Q(f,f)dy\right)d\eta dt \right|\\
                &\quad\quad\lesssim\norm{w_0(v)\Delta_j f_0(\eta,v)}^p_{L_x^p}  \norm{w_0(u)\Delta_k g_0(\eta,u)}^p_{L_{x}^p} + \norm{w(v)\Delta_j f(T,v)}^p_{L_x^p}\norm{w(u)\Delta_k g(T,u)}^p_{L_x^p}\\
                &\quad\quad+\int_{0}^{T} \norm{w(v)\Delta_j f(v)}^p_{L_{x}^p}\norm{w(u)\Delta_k g(u)}^{p-1}_{L_{x}^p}\norm{w(u)\Delta_k Q(g,g)}_{L_{x}^p}dt\\
                &\quad\quad+\int_{0}^{T} \norm{w(u)\Delta_k g(u)}^p_{L_{x}^p}\norm{w(v)\Delta_j f(v)}^{p-1}_{L_{x}^p}\norm{w(v)\Delta_j Q(f,f)}_{L_{x}^p}dt,
        \end{align*}
this together with $v_1-u_1=|v-u|$
further implies
            \begin{align*}
               &\norm{\moc{v-u}^{\frac{1}{p}}\norm{w(v)\Delta_j f(v)}_{L_y^{p}}\norm{w(u)\Delta_k g(u)}_{L_y^{p}}}^p_{L^p_{t,\eta}}\\
               &\quad\quad=\int_{\mathbb{R}}\int_{\mathbb{R}}\moc{v-u}\left(\int_{\mathbb{R}^2} \moc{w(v)\Delta_j f(v)}^p dy\right)\left(\int_{\mathbb{R}^2}\moc{w(u)\Delta_k g(u)}^pdy\right)dtd\eta\\
               &\quad\quad=\moc{\int_{\mathbb{R}}\int_{\mathbb{R}}(v_1-u_1)\left(\int_{\mathbb{R}^2} \moc{w(v)\Delta_j f(v)}^p dy\right)\left(\int_{\mathbb{R}^2}\moc{w(u)\Delta_k g(u)}^pdy\right)dtd\eta}\\
                &\quad\quad\lesssim\norm{ w_0(v)\Delta_j f_0(v) }^p_{L_x^p}\cdot \norm{w_0(u)\Delta_k g_0(u)}^p_{L_{x}^p} + \norm{w(v)\Delta_j f(T,v)}^p_{L_x^p}\norm{w(u)\Delta_k g(T,u)}^p_{L_x^p}\\
                &\quad\quad\quad+\int_{0}^{T} \norm{w(v)\Delta_j f(v)}^p_{L_{x}^p}\norm{w(u)\Delta_k g(u)}^{p-1}_{L_{x}^p}\norm{w(u)\Delta_k Q(g,g)}_{L_{x}^p}dt\\
                &\quad\quad\quad+\int_{0}^{T} \norm{w(u)\Delta_k g(u)}^p_{L_{x}^p}\norm{w(v)\Delta_j f(v)}^{p-1}_{L_{x}^p}\norm{w(v)\Delta_j Q(f,f)}_{L_{x}^p}dt.
            \end{align*}
Consequently, it follows
            \begin{align*}
               &\norm{\moc{v-u}^{\frac{1}{p}}\norm{w(v)\Delta_j f(v)}_{L_y^{p}}\norm{w(u)\Delta_k g(u)}_{L_y^{p}}}_{L^p_{t,u,v,\eta}}\\
               &\lesssim \norm{w_0(v)\Delta_j f_0(v)}_{L_{v,x}^p}\norm{w_0(u)\Delta_j g_0(u)}_{L_{u,x}^p} + \norm{w(v)\Delta_j f(T,v)}_{L_{v,x}^p}\norm{w(u)\Delta_k g(T,u)}_{L_{u,x}^p}\\
               &+\left(\int_{0}^{T} \norm{w(v)\Delta_j f(v)}^p_{L_{v,x}^p}\norm{w(u)\Delta_k g(u)}^{p-1}_{L_{u,x}^p}\norm{w(u)\Delta_k Q(g,g)}_{L_{u,x}^p}dt\right)^{\frac{1}{p}}\\
                &+\left(\int_{0}^{T} \norm{w(u)\Delta_k g(u)}^p_{L_{u,x}^p}\norm{w(v)\Delta_j f(v)}^{p-1}_{L_{v,x}^p}\norm{w(v)\Delta_j Q(f,f)}_{L_{v,x}^p}dt\right)^{\frac{1}{p}}\\
                &\lesssim \norm{w_0(v)\Delta_j f_0(v)}_{L_{v,x}^p}\norm{w_0(u)\Delta_k g_0(u)}_{L_{u,x}^p} + \norm{w(v)\Delta_j f(T,v)}_{L_{v,x}^p}\norm{w(u)\Delta_k g(T,u)}_{L_{u,x}^p}\\
               &+\norm{w(v)\Delta_j f(v)}_{L_{T}^{\infty}L_{v,x}^p}\norm{w(u)\Delta_k g(u)}^{1-\frac{1}{p}}_{L_{T}^{\infty}L_{u,x}^p} \left(\int_{0}^{T} \norm{w(u)\Delta_k Q(g,g)}_{L^p_{u,x}}dt \right)^{\frac{1}{p}}   \\
                &+ \norm{w(u)\Delta_k g(u)}_{L_{T}^{\infty}L_{u,x}^p}\norm{w(v)\Delta_j f(v)}^{1-\frac{1}{p}}_{L_{T}^{\infty}L_{v,x}^p}
                 \left(\int_{0}^{T} \norm{w(v)\Delta_j Q(f,f)}_{L^p_{v,x}}dt \right)^{\frac{1}{p}}
            \end{align*}
                    Recalling the definitions \eqref{Energy1}, \eqref{Energy1-T=0} and \eqref{D_T(f,g)}, and applying H\"{o}lder's inequality
                    , we get
             \begin{align}
               \mathscr{D}_T(f,g)&\lesssim \mathscr{E}_0(f)\mathscr{E}_0(g) +  \mathscr{E}_T(f)\mathscr{E}_T(g) \notag\\
               &\quad\quad+ \mathscr{E}_T(f)\left(\sum_{k\geq -1}2^{\frac{2k}{p}}\norm{w(u)\Delta_k g(u)}^{\frac{p-1}{p}}_{L_{T}^{\infty}L_{u,x}^p}\left(\int_{0}^{T} \norm{w(u)\Delta_j Q(g,g)}_{L_{u,x}^p}dt \right)^{\frac{1}{p}}  \right) \notag\\
               &\quad\quad+\mathscr{E}_T(g)\left(\sum_{j\geq -1}2^{\frac{2j}{p}}\norm{w(v)\Delta_j f(v)}^{\frac{p-1}{p}}_{L_{T}^{\infty}L_{v,x}^p}\left(\int_{0}^{T} \norm{w(v)\Delta_j Q(f,f)}_{L_{v,x}^p}dt\right)^{\frac{1}{p}}  \right)\notag\\
               & \lesssim \mathscr{E}_0(f)\mathscr{E}_0(g) +  \mathscr{E}_T(f)\mathscr{E}_T(g)
               + \mathscr{E}_T(f)\mathscr{E}^{1-\frac{1}{p}}_T(g)\left(\int_{0}^{T}\sum_{k\geq -1}2^{\frac{2k}{p}} \norm{w(u)\Delta_k Q(g,g)}_{L_{u,x}^p}dt\right)^{\frac{1}{p}} \notag \\
               &\quad\quad+\mathscr{E}_T(g)\mathscr{E}^{1-\frac{1}{p}}_T(f) \left(\int_{0}^{T}\sum_{j\geq -1}2^{\frac{2j}{p}} \norm{w(v)\Delta_j Q(f,f)}_{L_{v,x}^p}dt \right)^{\frac{1}{p}}.\label{bi-diss-es}
             \end{align}
Finally, combining \eqref{bi-diss-es} with \eqref{Estimate of Q(f,f)}, we obtain \eqref{Estimation of D_T(f,g)}.
The estimate \eqref{Estimation of D_T(f,f)} then follows immediately as a direct consequence of \eqref{Estimation of D_T(f,g)}.
This completes the proof of Proposition \ref{Estimation of D_T(f,g) prop}.
        \end{proof}
With Proposition \ref{Estimation of D_T(f,g) prop} at hand, we now derive the global {\it a priori} energy-type estimates for \eqref{Change variable BE}, under the {\it a priori} assumption that
\begin{align}\label{aps}
\mathscr{E}^2_T(f)+ \mathscr{D}_T(f) \le \varepsilon_0,
\end{align}
for all $T \ge 0$, where $\varepsilon_0>0$ is sufficiently small. For results in this direction, we have the following theorem.

\begin{theorem}\label{L^p Energy Estimate for}
Assume that $f$ is a strong solution to \eqref{Change variable BE} satisfying \eqref{aps}. Then, for $p\in (1,+\infty)$ with $\gamma\in \left(\dfrac{3}{p}-2,\dfrac{2}{p} -1\right)$ or $p=1$ with $\gamma = 1$, it holds that
\begin{equation}\label{priori estimate 2}
\mathscr{E}^2_T(f) + \mathscr{D}_T(f)
\lesssim \mathscr{E}^2_0(f).
\end{equation}
\end{theorem}

        \begin{proof}
          Note that
\begin{equation}
(\partial_t + v \cdot \nabla_x)\big(w(v)\Delta_j f\big)
= w(v)\Delta_j Q(f,f).
\notag
\end{equation}
Multiplying the above identity by
$p|w(v)\Delta_j f|^{p-2}w(v)\Delta_j f$
and integrating over $(x,v)\in \mathbb{R}^3\times\mathbb{R}^3$, we obtain
\begin{align*}
\partial_t \| w(v)\Delta_j f \|_{L_x^p}^p
= p \int_{\mathbb{R}^3} |w(v)\Delta_j f|^{p-2} w(v)\Delta_j f \, w(v)\Delta_j Q(f,f)\, dx.
\end{align*}
Integrating in time and using H\"{o}lder's inequality, it follows that
\begin{align*}
\| w(v)\Delta_j f \|_{L_{v,x}^p}
&\lesssim \| w_0(v)\Delta_j f_0 \|_{L_{v,x}^p} \\
&\quad + \| w(v)\Delta_j f \|_{L_T^\infty L_{v,x}^p}^{\frac{p-1}{p}}
\left( \int_0^T \| w(v)\Delta_j Q(f,f) \|_{L_{v,x}^p} \, dt \right)^{\frac{1}{p}}.
\end{align*}

Summing over $j\ge -1$ with weights $2^{\frac{2j}{p}}$, and applying H\"{o}lder's inequality together with Proposition \ref{Estimate of Q(f,f)}, we deduce
\begin{align}\label{ENERGY ESTIMATE}
\mathscr{E}_T(f)
&= \sum_{j\ge -1} 2^{\frac{2j}{p}} \| w(v)\Delta_j f \|_{L_T^\infty L_{v,x}^p} \notag\\
&\lesssim \mathscr{E}_0(f)
+ \mathscr{E}_T^{1-\frac{1}{p}}(f)
\left( \sum_{j\ge -1} 2^{\frac{2j}{p}} \int_0^T \| w(v)\Delta_j Q(f,f) \|_{L_{v,x}^p} \, dt \right)^{\frac{1}{p}} \notag\\
&\lesssim \mathscr{E}_0(f) + \mathscr{E}_T^{1-\frac{1}{p}}(f)\mathscr{D}_T^{\frac{1}{p}}(f).
\end{align}

Combining this with \eqref{Estimation of D_T(f,f)} in Proposition \ref{Estimation of D_T(f,g) prop}, we obtain
\begin{equation}\label{priori estimate to}
\mathscr{E}_T^2(f) + \mathscr{D}_T(f)
\lesssim \mathscr{E}_0^2(f)
+ \mathscr{E}_T^{2-\frac{2}{p}}(f)\mathscr{D}_T^{\frac{2}{p}}(f)
+ \mathscr{E}_T^{2-\frac{1}{p}}(f)\mathscr{D}_T^{\frac{1}{p}}(f).
\end{equation}

Under the {\it a priori} assumption \eqref{aps}, it follows that \eqref{priori estimate to} yields \eqref{priori estimate 2}.
This completes the proof of Theorem \ref{L^p Energy Estimate for}.
\end{proof}

    \section{The global existence}\label{sec-glo}
        In this section, we will establish the global-in-time existence of solution to the Boltzmann equation \eqref{Change variable BE} for all $t>0$. The construction of the global solution is based on the global {\it a priori} established in Theorem \ref{L^p Energy Estimate for} in Section \ref{ape-sec} and a uniform bilinear estimate for the following sequence of iterating approximate solutions:
        \begin{equation}\label{approximate equation-BE}
          \left\{\begin{array}{rl}
          \partial_t f^{n+1} + v\cdot\nabla_x f^{n+1} + f^{n+1}(v)&\int_{\mathbb{R}^3\times \mathbb{S}^2} \moc{v-u}^{\gamma}\mathbf{b}(\cos\theta) f^{n}(u)dud\omega
                                                        \\
                                                        &=\int_{\mathbb{R}^3\times \mathbb{S}^2} \moc{v-u}^{\gamma}\mathbf{b}(\cos\theta) f^n(v')f^{n}(u')dud\omega,  \\
           f^{n+1}(0,x,v)&=f_0(x,v),\ n\geq0,
        \end{array}\right.
        \end{equation}
        starting with $f^0(t,x,v) = 0$.

        \begin{lemma}\label{ESTIMATE OF M_0}
          The solution sequence $\{f^n\}_{n=1}^{\infty}$ is well defined. For $p\in (1,+\infty)$ with $\gamma\in \left(\dfrac{3}{p}-2,\dfrac{2}{p} -1\right)$ or $p=1$ with $\gamma = 1$. Then there exists a sufficiently small constant $\vps>0$ such that if
          $$\mathscr{E}_0(f_0)=\norm{f_0}_{\tilde{L}_{v,\eta}^p(B_{p,1,w_0}^{2/p})} \leq \vps, $$
     it holds for any $T\geq0$ and all $n\in\mathbb{N}$ that
          \begin{equation}\label{Y_T(f)'s estimate.}
            \mathscr{E}_T(f^n) \leq C\vps,\  \mathscr{D}_T(f^n)\leq C\vps^2,\quad\quad \forall T\in \mathbb{R}^{+},
          \end{equation}
     where $C>0$ is a constant independent of $n$.
        \end{lemma}

        \begin{proof}
We argue by induction on $n$. More precisely, we aim to prove the strengthened estimate
\begin{equation}\label{STRENGTHENING INDUCTION}
\mathscr{E}_T(f^i)\leq C\vps,\quad
\mathscr{D}_T(f^i)\leq C\vps^2,\quad
\mathscr{D}_T(f^i,f^{i-1})\leq C\vps^2,
\end{equation}
for all $i\geq 1$, where $C>0$ is independent of $i$.

We start with the proof of \eqref{STRENGTHENING INDUCTION} with $i=1$, since $f^0=0$, then $f^1$ satisfies
         \begin{equation}\label{equation of f^1}
           \partial_t f^1 + v\cdot\nabla_x f^1=0,
         \end{equation}
        with $f^1(0,x,v) = f_0(x,v)$. Using \eqref{equation of f^1}, direct calculation gives
$$ \mathscr{E}_T(f^1) =\norm{f_0}_{\tilde{L}_{v,\eta}^p(B_{p,1,w_0}^{2/p})}\leq \vps.$$

Now consider the system
        $$
        \left\{
        \begin{array}{rcl}
          \partial_t \int_{\mathbb{R}^2} \moc{w(v)\Delta_j f^1(v)}^{p} dy + v_1\cdot \partial_{\eta}\int_{\mathbb{R}^2} \moc{w(v)\Delta_j f^1(v)}^{p}dy&=&0,\\[2ex]
          \partial_t \int_{\mathbb{R}^2} \moc{w(u)\Delta_k f^1(u)}^{p} dy + u_1\cdot \partial_{\eta}\int_{\mathbb{R}^2} \moc{w(u)\Delta_k f^1(u)}^{p}dy&=&0,
        \end{array}
        \right.
        $$
 as deriving \eqref{Estimation of D_T(f,g)} in Proposition \ref{Estimation of D_T(f,g) prop}, one has
        \begin{equation}
          \mathscr{D}_T(f^1)\lesssim \mathscr{E}^2_0(f_0) + \mathscr{E}_T^2(f^1) \lesssim \mathscr{E}^2_0(f_0),\notag
        \end{equation}
and it is clear $ \mathscr{D}_{T}(f^{1},f^0) = 0\leq C\vps^2$ due to $f^0=0.$

Now we assume that \eqref{STRENGTHENING INDUCTION} holds for all $i\leq n$. We prove it for $i=n+1$.
Note that in the following calculations, we make an {\it a priori} assumption that the quantities in \eqref{STRENGTHENING INDUCTION} are bounded for $i=n+1$.

Using $\eqref{approximate equation-BE}_1$, we obtain
          \begin{align}
          \partial_t&\norm{w(v)\Delta_j f^{n+1}(v)}_{L_x^p}\notag\\
           =& p\int_{\mathbb{R}^3}\moc{w(v)\Delta_j f^{n+1}(v)}^{p-2}w(v)\Delta_j f^{n+1}(v)w(v)\Delta_j(Q_{\textrm{gain}}(f^n,f^n) - Q_{\textrm{loss}}(f^{n+1},f^n))dx,\notag
          \end{align}
which further gives
          \begin{align*}
            \norm{w(v)\Delta_j f^{n+1}(v)}_{L_{v,x}^p}&\lesssim \norm{w_0(v)\Delta_j f_{0}(v)}_{L_{v,x}^p}+ \norm{w(v)\Delta_j f^{n+1}(v)}_{L_T^{\infty}L_{v,x}^p}^{\frac{p-1}{p}}\\
            &\qquad\qquad\times\left(\int_{0}^{T} \norm{w(v)\Delta_j (Q_{\textrm{gain}}(f^n,f^n)-Q_{\textrm{loss}}(f^{n+1},f^{n}))}_{L_{vx}^p} dt \right)^{\frac{1}{p}}.
\end{align*}
In view of  H\"{o}lder's inequality and \eqref{Estimate of Q(f,f)}, we get
        \begin{align}\label{Ef_n+1 Estimate}
          \notag  \mathscr{E}_T(f^{n+1})&=\sum_{j\geq - 1} 2^{\frac{2j}{p}} \norm{w(v)\Delta_j f^{n}}_{L_T^{\infty}L_{v,x}^p}\notag\\
                                  &\lesssim \sum_{j\geq - 1} 2^{\frac{2j}{p}} \norm{w_0(v)\Delta_j f_{0}}_{L_{v,x}^p} + \sum_{j\geq - 1} 2^{\frac{2j}{p}}\norm{w(v)\Delta_j f^{n+1}(v)}_{L_T^{\infty}L_{v,x}^p}^{\frac{p-1}{p}}\notag\\
                                  &\quad\quad\times\left(\int_{0}^{T} \norm{w(v)\Delta_j (Q_{\textrm{gain}}(f^n,f^n)-Q_{\textrm{loss}}(f^{n+1},f^{n}))}_{L_{v,x}^p} dt \right)^{\frac{1}{p}}\notag\\
                                  &\lesssim \norm{f_0}_{\tilde{L}_{v,\eta}^p(B_{p,1,w_0}^{2/p})}\notag\\
                                  &\qquad+ \mathscr{E}_{T}^{1-\frac{1}{p}}(f^{n+1})\left(\sum_{j\geq-1}2^{\frac{2j}{p}}\int_{0}^{T} \norm{w(v)\Delta_j(Q_{\textrm{gain}}(f^n,f^n)-Q_{\textrm{loss}}(f^{n+1},f^{n}))}_{L_{v,x}^p} dt \right)^{\frac{1}{p}}\notag\\
                                  &\lesssim \norm{f_0}_{\tilde{L}_{v,\eta}^p(B_{p,1,w_0}^{2/p})} +  \mathscr{E}_{T}^{1-\frac{1}{p}}(f^{n+1})(\mathscr{D}_T(f^n,f^{n+1})+\mathscr{D}_T(f^n,f^n) + \mathscr{D}_T(f^{n+1},f^n))^{\frac{1}{p}}\notag\\
                                  &\lesssim \norm{f_0}_{\tilde{L}_{v,\eta}^p(B_{p,1,w_0}^{2/p})} +  \mathscr{E}_{T}^{1-\frac{1}{p}}(f^{n+1})(\mathscr{D}_T(f^n,f^{n+1})+\mathscr{D}_T(f^n) )^{\frac{1}{p}}\notag\\
                                  &\lesssim \vps +  \mathscr{E}_{T}^{1-\frac{1}{p}}(f^{n+1})(\mathscr{D}_T(f^{n+1},f^n)+\vps^2 )^{\frac{1}{p}}.
          \end{align}
On the other hand, as deriving \eqref{Estimation of D_T(f,g)} and using the induction hypothesis, we obtain
           \begin{align*}
               \mathscr{D}_T(f^{n+1},f^n)
               &\lesssim \norm{f_0}_{\tilde{L}_{v,\eta}^p(B_{p,1,w_0}^{2/p})}^2 +  \mathscr{E}_T(f^{n+1})\mathscr{E}_T(f^n)
               + \mathscr{E}_T(f^{n+1})\mathscr{E}^{1-\frac{1}{p}}_T(f^{n})(\mathscr{D}_T(f^{n},f^{n-1})
               \\
               &\quad+\mathscr{D}_T(f^{n-1}))^{\frac{1}{p}}  +\mathscr{E}_T(f^{n})\mathscr{E}^{1-\frac{1}{p}}_T(f^{n+1})(\mathscr{D}_T(f^{n+1},f^{n}) + \mathscr{D}_T(f^{n}))^{\frac{1}{p}} \\
               &\lesssim \vps^2 + \vps\mathscr{E}_T(f^{n+1}) + \vps^{1+\frac{2}{p}}\mathscr{E}^{1-\frac{1}{p}}_T(f^{n+1}) + \vps\mathscr{E}_T^{1-\frac{1}{p}}(f^{n+1})\mathscr{D}^{\frac{1}{p}}_T(f^{n+1},f^{n}) \\
               &\lesssim \vps^2 + \vps\mathscr{E}_T(f^{n+1}) + \vps\mathscr{D}_T(f^{n+1},f^{n}),
             \end{align*}
     hence, for $\vps$ suitably small, it follows
          \begin{equation}\label{ Estimate of D_T(f_n+1,f_n)}
            \mathscr{D}_T(f_{n+1},f_n)\lesssim \vps^2 + \vps \mathscr{E}_T(f_{n+1}).
          \end{equation}
\eqref{ Estimate of D_T(f_n+1,f_n)} together with \eqref{Ef_n+1 Estimate} gives
          \begin{align*}
            \mathscr{E}_T(f^{n+1})&\lesssim \vps + \mathscr{E}_{T}^{1-\frac{1}{p}}(f^{n+1})(\vps\mathscr{E}_T(f^{n+1})+\vps^2 )^{\frac{1}{p}}\lesssim \vps +\vps^{\frac{1}{p}}\mathscr{E}_{T}(f^{n+1}) + \vps^{1+\frac{1}{p}},
          \end{align*}
          therefore
          \begin{equation}
            \mathscr{E}_T(f^{n+1})\leq C\vps,\notag
          \end{equation}
          and
          \begin{equation}
            \mathscr{D}_T(f^{n+1},f^n)\lesssim \vps^2 + \vps \mathscr{E}_T(f^{n+1})\leq \vps^2.\notag
          \end{equation}
          Finally, applying Proposition \ref{Estimation of D_T(f,g) prop} once more, we obtain
          \begin{align}
            \mathscr{D}_T(f^{n+1})&\leq \norm{f_0}^2_{\tilde{L}_{v,\eta}^p(B_{p,1,w_0}^{2/p})} + \mathscr{E}_T^2(f^{n+1}) +\mathscr{E}_T^{2-\frac{1}{p}}(f^{n+1})\left( \mathscr{D}_T(f^{n+1},f^{n}) + \mathscr{D}_T(f^{n}) \right)^{\frac{1}{p}}\lesssim \vps^2.\notag
          \end{align}
          By induction, \eqref{Y_T(f)'s estimate.} holds for any $n\geq0$, this finishes the proof of Lemma \ref{ESTIMATE OF M_0}.

        \end{proof}
       With the uniform bounds on the iterative solution sequence established in Lemma \ref{ESTIMATE OF M_0} for the approximate system \eqref{approximate equation-BE}, we are now in a position to prove the global existence for the Cauchy problem \eqref{Change variable BE}.


\begin{proof}[Proof of Theorem \ref{Well-posedeness}]
By the construction of the approximate sequence $\{f^n\}_{n\geq 1}$ through \eqref{approximate equation-BE} and the uniform bounds established in Lemma \ref{ESTIMATE OF M_0}, one can pass to the limit to obtain a function $f(t,x,v)$, which is a distributional solution to \eqref{Change variable BE} with initial data $f(0,x,v)=f_0(x,v)$.

Moreover, the estimate \eqref{eng-es} follows directly from the global \textit{a priori} estimate \eqref{priori estimate 2}. In particular, these bounds imply sufficient regularity so that the distributional solution is indeed a strong solution.

Next we show the norm $\mathscr{E}_{t}(f)$ defined in \eqref{Energy1} is continuous in $t$ variable.
For all $t,t_0\in [0,T]$ with $T>0$, similar to obtain \eqref{ENERGY ESTIMATE}, we have
\begin{equation}\label{E_t_1 - E_t_2}
            \moc{\mathscr{E}_{t}(f) - \mathscr{E}_{t_0}(f)}\lesssim \vps^{1-\frac{1}{p}}\left(\sum_{j\geq -1} 2^{\frac{2j}{p}} \int_{t_0}^{t} \norm{w(v)\Delta_j Q(f,f)}_{L_{v,x}^p}ds \right)^{\frac{1}{p}},
\end{equation}
With this, it suffices to prove
 \begin{equation}\label{lim_t Qff}
  \lim_{t\to t_0}\left(\sum_{j\geq -1} 2^{\frac{2j}{p}} \int_{t_0}^{t} \norm{w(v)\Delta_j Q(f,f)}_{L_{v,x}^p}ds\right)^{\frac{1}{p}} = 0.
  \end{equation}
  By
  $$ \left(\sum_{j\geq -1} 2^{\frac{2j}{p}} \int_{t_0}^{t} \norm{w(v)\Delta_j Q(f,f)}_{L_{v,x}^p}ds\right)^{\frac{1}{p}}\lesssim \mathscr{D}_T(f)\leq C\mathscr{E}^2_{0}(f), $$
which implies, $\forall \eps > 0,$ there exists a $N>0,$ such that
          $$ \sum_{j\geq N} 2^{\frac{2j}{p}} \int_{t_0}^{t} \norm{w(v)\Delta_j Q(f,f)}_{L_{v,x}^p}ds \leq \eps, $$
          and by absolute continuity of integrals $$\lim_{t\to t_0} \sum_{-1\leq j\leq N-1} 2^{\frac{2j}{p}} \int_{t_0}^{t} \norm{w(v)\Delta_j Q(f,f)}_{L_{v,x}^p}ds=0.$$
Therefore there exists $\delta>0,$ such that $\forall t\in (t_0 - \delta,t_0+\delta)\cap [0,T]$,
          $$\sum_{-1\leq j\leq N-1} 2^{\frac{2j}{p}} \int_{t_0}^{t} \norm{w(v)\Delta_j Q(f,f)}_{L_{v,x}^p}ds\leq \eps.$$
Consequently, it follows
          $$ \sum_{j\geq -1} 2^{\frac{2j}{p}} \int_{t_0}^{t} \norm{w(v)\Delta_j Q(f,f)}_{L_{v,x}^p}ds\leq 2\eps, $$
 then \eqref{lim_t Qff} is valid. Thus $\mathscr{E}_{t}(f)$ is continuous due to \eqref{E_t_1 - E_t_2}.

We now turn to the uniqueness. Suppose that $g$ is another solution to \eqref{Change variable BE} with the same initial data $g(0,x,v)=f_0(x,v)$, satisfying
           $$ \mathscr{E}_T(g)\leq C\vps,\  \mathscr{D}_T(g)\leq C\vps^2, $$
           for any $T\in [0,+\infty).$ Taking the difference of Bolzmann equation \eqref{Change variable BE} for $f$ and $g$, one has
           $$ \left(\partial_t+v\cdot \nabla_x \right)(f-g) = Q(f-g,f)+ Q(g,f-g),$$
           Then, by performing the completely same energy estimate as in Theorem \ref{L^p Energy Estimate for}, it follows that
           \begin{align*}
            \mathscr{E}_T(f-g)&=\sum_{j\geq - 1} 2^{\frac{2j}{p}} \norm{w(v)\Delta_j (f-g)}_{L_T^{\infty}L_{v,x}^p}\\
                              &\lesssim \mathscr{E}_{T}(f-g)^{1-\frac{1}{p}}\left(\sum_{j\geq -1}2^{\frac{2j}{p}} \int_{0}^{T} \norm{w(v)\Delta_j(Q(f-g,f)+ Q(g,f-g))}_{L_{v,x}^p}\right)^{\frac{1}{p}}\\
                              &\lesssim \mathscr{E}_{T}(f-g)^{1-\frac{1}{p}}\left(\mathscr{D}_T(f-g,f)+\mathscr{D}_T(g,f-g)\right)^{\frac{1}{p}},
          \end{align*}
          thus
          \begin{equation}\label{E_T(f-g)}
            \mathscr{E}_T(f-g)\lesssim \mathscr{D}_T(f-g,f)+\mathscr{D}_T(g,f-g).
          \end{equation}
On the other hand, arguing as in the proof of Proposition \ref{Estimation of D_T(f,g) prop}, we deduce
           \begin{align}
            \mathscr{D}_T(f-g,f) \lesssim&  \mathscr{E}_T(f-g)\mathscr{E}_T(f) + \mathscr{E}_T(f-g)\mathscr{E}^{1-\frac{1}{p}}_T(f)\mathscr{D}^{\frac{1}{p}}_T(f)
           \notag\\& +\mathscr{E}_T(f)\mathscr{E}^{1-\frac{1}{p}}_T(f-g)\left(\mathscr{D}_T(f-g,f)
            + \mathscr{D}_T(g,f-g)\right)^{\frac{1}{p}},\notag
          \end{align}
          and
          \begin{align}
            \mathscr{D}_T(g,f-g) \lesssim&  \mathscr{E}_T(f-g)\mathscr{E}_T(g) + \mathscr{E}_T(f-g)\mathscr{E}^{1-\frac{1}{p}}_T(g)\mathscr{D}^{\frac{1}{p}}_T(g)
            \notag\\&+\mathscr{E}_T(g)\mathscr{E}^{1-\frac{1}{p}}_T(f-g)\left(\mathscr{D}_T(f-g,f)
            + \mathscr{D}_T(g,f-g)\right)^{\frac{1}{p}}.\notag
          \end{align}
Using the uniform bounds on $f$ and $g$, we infer
          \begin{align*}
             \mathscr{D}_T(f-g,f)+\mathscr{D}_T(g,f-g)
             \lesssim & (\vps+\vps^{1+\frac{1}{p}})\mathscr{E}_T(f-g)+\vps^{1+\frac{2}{p}}\mathscr{E}_T^{1-\frac{1}{p}}(f-g)
             \\&+\vps\mathscr{E}^{1-\frac{1}{p}}_T(f-g)
             (\mathscr{D}_T(f-g,f)+\mathscr{D}_T(g,f-g))^{\frac{1}{p}}.
          \end{align*}
For $\vps$ sufficiently small, this implies
          \begin{equation}\label{D_T(f-g,f)+D_T(g,f-g)}
            \mathscr{D}_T(f-g,f) + \mathscr{D}_T(g,f-g)\lesssim \vps\mathscr{E}_T(f-g).
          \end{equation}
      Combining \eqref{E_T(f-g)} and \eqref{D_T(f-g,f)+D_T(g,f-g)}, we obtain
          $$  \mathscr{E}_T(f-g)\lesssim  \vps\mathscr{E}_{T}(f-g), $$
          then $\mathscr{E}_T(f-g) = 0,$ and $f\equiv g.$

Finally, the non-negativity of solutions follows from the iterative scheme \eqref{approximate equation-BE} by a standard argument, and we omit the details.  The proof of Theorem \ref{Well-posedeness} is complete.

\end{proof}

    \section{Appendix}\label{sec-app}

In this appendix, we collect several key estimates and lemmas that have been used throughout the previous sections. These include basic integral estimates, bounds on gain and loss terms, the boundedness of a bilinear summation operator on sequence spaces, and the div-curl lemma. Each result is stated clearly for reference and, where appropriate, proofs are provided or outlined.

The following lemma provides a basic integral estimate.
\begin{lemma}\label{Integral Estimate Lemma}
Let $l_1,l_2\in \mathbb{R}$ satisfy $0 \leq l_1 < n$, $l_2>0$, and $l_1+l_2>n$. Then for any $s\in \mathbb{R}^n$, it holds that
\[
\int_{\mathbb{R}^n} \frac{dw}{|w|^{l_1}\langle w - s \rangle^{l_2}}
\lesssim \frac{1}{\langle s \rangle^{l_1+l_2-n}}
\lesssim 1.
\]
\end{lemma}

\begin{proof}
The proof is elementary and thus omitted.
\end{proof}
The next two lemmas follow directly from Lemma \ref{Integral Estimate Lemma}.
        \begin{lemma}[Estimate of $I_{\textrm{loss}}$]\label{lem-loss}
Let $l>\max\left\{ 1-\dfrac{1}{p},3-\dfrac{4}{p}+\gamma\right\},$ then for $\dfrac{4}{p}-3< \gamma\leq \dfrac{1}{p}$ and $p\in (1,+\infty)$, it holds that
          \begin{align}\label{Estimate of I_loss}
           I_{\textrm{loss}}= \norm{\moc{v-u}^{\gamma - \frac{1}{p}}w^{-1}(u)}_{L^{p'}_u}\lesssim (t+1)^{-\gamma+\frac{1}{p}-\frac{3}{p'}}.
          \end{align}
        \end{lemma}

        \begin{proof}
Note that $\left(-\gamma+\dfrac{1}{p}+l\right)p'> 3, $ $ 0\leq\left(-\gamma + \dfrac{1}{p} \right)p' < 3,$ and $ lp'>1,$ it follows from Lemma \ref{Integral Estimate Lemma} that
          \begin{align*}
            I_{\textrm{loss}}^{p'}&=\int_{\mathbb{R}^3} \dfrac{du}{\moc{v-u}^{(-\gamma+\frac{1}{p})p'}\langle x - (t+1)u \rangle^{lp'}}\\
                         &=\int_{\mathbb{R}^3} \dfrac{du}{\moc{v-u}^{(-\gamma+\frac{1}{p})p'}\langle x - (t+1)w \rangle^{lp'}}\\
                         &\xlongequal{u:=u-v,\ s = x-(t+1)v} \int_{\mathbb{R}^3} \dfrac{du}{\moc{u}^{(-\gamma+\frac{1}{p})p'}\langle s - (t+1)u \rangle^{lp'}}\\
                         &\xlongequal{u:=\frac{u}{t+1}} \int_{\mathbb{R}^3}(t+1)^{(-\gamma+\frac{1}{p})p'-3} \dfrac{du}{\moc{u}^{(-\gamma+\frac{1}{p})p'}\langle s - u \rangle^{b}}\lesssim (t+1)^{(-\gamma+\frac{1}{p})p'-3},
          \end{align*}
which completes the proof of Lemma \ref{lem-loss}.
        \end{proof}

        \begin{lemma}[Estimate of $I_{\mathrm{gain}}$]\label{lem-gain}
          Let $l>\max\left\{ 1-\dfrac{1}{p},3-\dfrac{4}{p}+\gamma\right\},$ then for $\dfrac{4}{p}-3< \gamma <  \dfrac{2}{p}-1$ and $p\in (1,+\infty)$, it holds that
          \begin{equation}\label{Estimate of I_gain}
            I_{\mathrm{gain}}= \norm{\moc{v-u}^{\gamma - \frac{1}{p}}w(v)w^{-1}(v')w^{-1}(u')}_{L^{p'}_u}\lesssim (t+1)^{-\gamma+\frac{1}{p}-\frac{3}{p'}}.
          \end{equation}
        \end{lemma}

        \begin{proof}
For simplicity, denote $ a = \left(-\gamma + \dfrac{1}{p} \right)p',$ and $b = lp'$, then we see that $a\in (1,3)$ and $b>1$. Without loss of general, we assume $b<3$. We compute
        \begin{align*}
            I_{\mathrm{gain}}^{p'}&=\int_{\mathbb{R}^3} \dfrac{\langle x-(t+1)v\rangle^{lp'}du}{\moc{v-u}^{(-\gamma+\frac{1}{p})p'}\langle x-(t+1)v'\rangle^{lp'}\langle x-(t+1)u'\rangle^{lp'}}\\
                         &\xlongequal{s = x-(t+1)v}\int_{\mathbb{R}^3} \dfrac{\langle s \rangle^{lp'}du}{\moc{v-u}^{(-\gamma+\frac{1}{p})p'}\langle x-(t+1)v'\rangle^{lp'}\langle x-(t+1)u'\rangle^{lp'}}\\
                         &=\int_{\mathbb{R}^3} \dfrac{\langle s \rangle^{b}du}{\moc{u-v}^{a}\langle x-(t+1)v + (t+1)(\omega\cdot (v-u))\omega\rangle^{b}\langle x-(t+1)(u+(\omega\cdot (v-u))\omega)\rangle^{b}}\\
                         &\xlongequal{u:=u-v} \int_{\mathbb{R}^3} \dfrac{\langle s \rangle^{b}du}{\moc{u}^{a}\langle s - (t+1)(\omega\cdot u)\omega\rangle^{b}\langle s-(t+1)(u-(\omega\cdot u )\omega)\rangle^{b}}\\
                         &\xlongequal{u:=\frac{u}{t+1}} (t+1)^{a-3}\int_{\mathbb{R}^3} \dfrac{\langle s \rangle^{b}du}{\moc{u}^{a}\langle s - (\omega\cdot u)\omega\rangle^{b}\langle s-(u-(\omega\cdot u)\omega)\rangle^{b}}:=(t+1)^{a-3}J.
          \end{align*}
It remains now to show $J$ is bounded uniformly with respect to $x,t$ and $v$.
          Since $(\omega\cdot u)\omega$ and $u-(\omega\cdot u)\omega$ are orthogonal, we introduce an orthogonal matrix $A$ such that
         \begin{align}\notag\left\{
          \begin{array}{rll}
            &Ae_1 = \omega, \\
            &Ae_2 = \dfrac{s-(s\cdot \omega)\omega}{\moc{s-(s\cdot \omega)\omega}} ,\\
            &Ae_3 \bot s,
          \end{array}
          \right.\end{align}
           thus $A^T \omega = e_1$ and $\omega \cdot Au = A^T\omega \cdot u = u_1.$ By performing the change of variables $u \mapsto u^* = Au$, we express the integral $J$ in the orthonormal basis $\{\omega, Ae_2, Ae_3\}$. Write $u^*=(u_1^*,u_2^*,u_3^*)$ and decompose
 $s = s_1 \omega + s_2 Ae_2 + 0\cdot Ae_3$, Then we have
          $$ \moc{s - (\omega\cdot u)\omega}^2 = s_2^2 + (s_1-u^{*}_1)^2,\ \moc{s-(u-(\omega\cdot u)\omega)}^2 = s_1^2 + (s_2 - u_2^{*})^2 + \moc{u_3^{*}}^2. $$
          Let $y = (u_1^*, u_2^*) \in \mathbb{R}^2$. By Fubini's theorem, we have
\begin{align*}
J
&= \int_{\mathbb{R}^3} \frac{\langle s \rangle^{b}\,du^*}{|u^*|^{a}(1+ s_2^2 + (s_1-u_1^*)^2 )^{\frac{b}{2}}(1+ s_1^2 + (s_2 - u_2^*)^2 + |u_3^*|^2 )^{\frac{b}{2}}} \\
&\lesssim \int_{\mathbb{R}^2} \frac{\langle s \rangle^{b}\,dy }{(1+ s_2^2 + (s_1-y_1)^2 )^{\frac{b}{2}}(1+ s_1^2 + (s_2 - y_2)^2 )^{\frac{b}{2}}}
\int_{\mathbb{R}} \frac{du_3^*}{(|y|^2 + |u_3^*|^2)^{\frac{a}{2}}}.
\end{align*}

Since $a>1$, the one-dimensional integral is finite. Indeed, by the change of variables $u_3^* = |y|\tan\theta$, we obtain
\begin{align*}
\int_{\mathbb{R}} \frac{du_3^*}{(|y|^2 + |u_3^*|^2)^{\frac{a}{2}}}
&= 2 \int_0^{\frac{\pi}{2}} \frac{|y|\sec^2\theta}{|y|^a \sec^a\theta}\, d\theta
= \frac{2}{|y|^{a-1}} \int_0^{\frac{\pi}{2}} \cos^{a-2}\theta \, d\theta \lesssim \frac{1}{|y|^{a-1}},
\end{align*}
where the last integral is finite since $a-2 \in (-1,1)$.

Therefore,
\[
J \lesssim \int_{\mathbb{R}^2}
\frac{\langle s \rangle^{b}\,dy}
{|y|^{a-1}(1+ s_2^2 + (s_1-y_1)^2 )^{\frac{b}{2}}(1+ s_1^2 + (s_2 - y_2)^2 )^{\frac{b}{2}}}.
\]
 The calculation for the right hand side above is then divided into two steps.

          \noindent \underline{\emph{Step 1: The case $|s|\geq 1$, i.e. $s_1^2+s_2^2\geq 1$.}}
\begin{align*}
J
&= \left(\int_{|y|\leq \frac{1}{2}\max\{|s_1|,|s_2|\}} + \int_{|y|\geq \frac{1}{2}\max\{|s_1|,|s_2|\}} \right)
\frac{\langle s \rangle^{b}\,dy}{|y|^{a-1}(1+ s_2^2 + (s_1-y_1)^2 )^{\frac{b}{2}}(1+ s_1^2 + (s_2 - y_2)^2 )^{\frac{b}{2}}} \\
&=: J_6 + J_7.
\end{align*}

\noindent
\textbf{Estimate of $J_6$.}
By symmetry in $y_1,y_2$, we may assume $|s_1|\geq |s_2|$, hence $|s_1|\geq \frac{\sqrt{2}}{2}$.
If $|y|\leq \frac{|s_1|}{2}$, then $|y_1-s_1|\geq \frac{|s_1|}{2}$. Thus
\begin{align*}
J_6
&\lesssim \int_{|y|\leq \frac{|s_1|}{2}} \frac{\langle s \rangle^b\,dy}{|y|^{a-1}(1+s_1^2+s_2^2)^{\frac{b}{2}} (1+s_1^2 + (y_2 - s_2)^2)^{\frac{b}{2}} } \\
&\lesssim \int_{|y|\leq \frac{|s_1|}{2}} \frac{dy}{|y|^{a-1}\langle s_1 \rangle^b}
\lesssim \frac{1}{\langle s_1 \rangle^b} \int_0^{\frac{|s_1|}{2}} r^{2-a}\,dr \\
&\lesssim \frac{1}{|s_1|^{a-3}\langle s_1 \rangle^b}
\lesssim \frac{1}{\langle s_1 \rangle^{a+b-3}}
\lesssim 1,
\end{align*}
since $a+b>3$.

\medskip
\noindent
\textbf{Estimate of $J_7$.}
We consider three separate cases.

\medskip
\noindent
\emph{Case 1: $a\in (1,2)$.}
\begin{align*}
J_7
&\lesssim \int_{|y|\geq \frac{|s_1|}{2}}
\frac{\langle s_1 \rangle^b\,dy}{|y_1|^{a-1}(1+(s_1-y_1)^2 +s_2^2)^{\frac{b}{2}} (1+s_1^2 + (y_2 - s_2)^2)^{\frac{b}{2}}} \\
&\lesssim \left(\int_{\mathbb{R}} \frac{\langle s_1 \rangle^b\,dy_1}{|y_1|^{a-1}(1+(s_1-y_1)^2 +s_2^2)^{\frac{b}{2}}}\right)
\left(\int_{\mathbb{R}} \frac{dy_2}{(1+s_1^2 + (y_2 - s_2)^2)^{\frac{b}{2}}}\right) \\
&=: J_{7,1}\cdot J_{7,2}.
\end{align*}
By Lemma \ref{Integral Estimate Lemma}, we have
$
J_{7,1} \lesssim \frac{1}{\langle s_1 \rangle^{a-2}}.
$
Moreover, by a standard one-dimensional estimate, it follows
$
J_{7,2} \lesssim \frac{1}{\langle s_1 \rangle^{b-1}} (b>1).
$
Therefore, we get
$
J_7 \lesssim \frac{1}{\langle s_1 \rangle^{a+b-3}} \lesssim 1.
$

\medskip
\noindent
\emph{Case 2: $a=2$.}
Choose $\tilde{a}\in(1,2)$ such that $\tilde{a}+b>3$. Since $|y|\geq \frac{|s_1|}{2}\gtrsim 1$, it follows that
$
|y|^{a-1} \gtrsim |y|^{\tilde{a}-1},
$
and thus the estimate reduces to Case 1. 

\medskip
\noindent
\emph{Case 3: $a\in(2,3)$.}
Using
\[
\int_{\mathbb{R}} \frac{dy_2}{(y_1^2+y_2^2)^{\frac{a-1}{2}}}
\lesssim \frac{1}{|y_1|^{a-2}},
\]
we obtain
\begin{align*}
J_7
&\lesssim \int_{\mathbb{R}} \frac{dy_1}{|y_1|^{a-2}(1+(s_1-y_1)^2 +s_2^2)^{\frac{b}{2}}}
\lesssim \frac{1}{\langle s_1 \rangle^{a+b-3}}
\lesssim 1.
\end{align*}

\medskip
\noindent\underline{\emph{Step 2: The case $|s|\leq 1$.}}

In this case $\langle s\rangle^b \lesssim 1$. Using $|y|^{a-1}\gtrsim |y_1|^{\frac{a-1}{2}}|y_2|^{\frac{a-1}{2}}$, we get
\begin{align*}
J
&\lesssim \int_{\mathbb{R}} \frac{dy_1}{|y_1|^{\frac{a-1}{2}}(1+(s_1-y_1)^2+s_2^2)^{\frac{b}{2}}}
\int_{\mathbb{R}} \frac{dy_2}{|y_2|^{\frac{a-1}{2}}(1+(s_2-y_2)^2+s_1^2)^{\frac{b}{2}}}.
\end{align*}
By Lemma \ref{Integral Estimate Lemma},
\[
\int_{\mathbb{R}} \frac{dy_i}{|y_i|^{\frac{a-1}{2}}(1+(s_i-y_i)^2}+s_k^2)^{\frac{b}{2}}
\lesssim \frac{1}{\langle s \rangle^{\frac{a-1}{2}+b-1}},
\]
and since $\frac{a-1}{2}+b-1>0$, we conclude
$
J \lesssim 1.
$

\medskip
\noindent
Combining the above estimates, we obtain $J\lesssim 1$, and hence
\[
I_{\mathrm{gain}}^{p'} \lesssim (t+1)^{a-3}J \lesssim (t+1)^{a-3}.
\]
This completes the proof of Lemma \ref{lem-gain}.
        \end{proof}
The following lemma concerns the boundedness of a summation operator on weighted $l^1$ space $l^1\left(2^{2(j+k)/p}\right)$.
        \begin{lemma}\label{Sum operator}
Define the weighted $l^1$ space
$$l^1\left(2^{2(j+k)/p}\right):= \left\{ x=\left(x_{j,k}\right)_{j,k\geq-1}\left|\quad \sum_{j,k\geq-1}2^{\frac{2(j+k)}{p}}\moc{x_{j,k}}<\infty\right. \right\} $$
with norm
$$ \norm{x}_{l^1(2^{2(j+k)/p})} = \sum_{j,k\geq-1}2^{\frac{2(j+k)}{p}}\moc{x_{j,k}}<\infty .$$
Define the operator
\[
\begin{array}{rccl}
\CT: & l^1\left(2^{2(j+k)/p}\right) & \longrightarrow & \mathbb{R}, \\[1ex]
     & x& \longmapsto & \mathcal{T}x,
\end{array}
\]
where
\[
\CT x := \sum_{j,k\geq -1} \Bigg( \sum_{|i-j|\leq 4} + \sum_{i\geq j-3} \Bigg) 2^{\frac{2(j+k)}{p}} \, x_{i,k},
\]
for sequences $x=\left( x_{j,k} \right)_{j,k\geq-1}$.

Then $\CT$ is bounded on $l^1\left(2^{ 2(j+k)/p}\right)$, i.e., there exists a constant $C>0$ (depending only on $p$) such that
\[
|\CT(x)| \leq C \norm{x}_{l^1(2^{2(j+k)/p})}.
\]
        \end{lemma}

        \begin{proof}
           For any sequence $x\in l^1(2^{2(j+k)/p})$, by Young's inequality on $l^{1}$, it follows
                    \begin{align*}
                      \sum_{j,k\geq -1}\left( \sum_{\moc{i-j}\leq 4}+\sum_{i\geq j-3}\right)2^{\frac{2(j+k)}{p}}x_{i,k}
                      =&\sum_{k\geq -1} 2^{\frac{2k}{p}} \sum_{j\geq-1}\left(\sum_{\moc{i-j}\leq 4}+\sum_{i\geq j-3} \right)2^{\frac{2(j-i)}{p}}2^{\frac{2i}{p}}x_{i,k}\\
                      =&\sum_{k\geq -1} 2^{\frac{2k}{p}}\sum_{j\geq-1}\left(\left[2^{-\frac{2i}{p}}\left( 1_{\moc{i}\leq 4}+1_{i\geq -3}\right) \right] *\left(2^{\frac{2i}{p}}x_{i,k}\right)\right)_{j}\\
                      \leq&  \norm{2^{-\frac{2j}{p}}\left( 1_{\moc{j}\leq 4}+1_{j\geq -3}\right)}_{l^1}\sum_{k\geq-1} 2^{\frac{2k}{p}} \sum_{j\geq-1}2^{\frac{2j}{p}}\moc{x_{j,k}}\\
                      \leq& C\norm{x}_{l^1(2^{2(j+k)/p})},
                    \end{align*}
                    thus the proof of Lemma \ref{Sum operator} is finished.
        \end{proof}

In what follows, we establish the weighted Besov embedding \eqref{Embedding inequailty}.

\begin{proof}[The proof of \eqref{Embedding inequailty}]
Since $l\geq0$, direct calculation yields
        	\begin{align*}
        		\norm{w(v)f}_{L_y^{\infty}}&=\norm{\langle x-(t+1)v\rangle^{l} f(t,x,v)}_{L_y^{\infty}}\\
        		&\lesssim \norm{\langle \eta-(t+1)v_1\rangle^{l} f(t,x,v)}_{L_y^{\infty}}
        		+\norm{\langle y-(t+1)v_y\rangle^{l} f(t,x,v)}_{L_y^{\infty}}\\
        		&= \langle \eta-(t+1)v_1\rangle^{l}\norm{ f(t,x,v)}_{L_y^{\infty}}
        		+\norm{\langle y\rangle^{l} f(t,\eta,y+(t+1)v_y,v)}_{L_y^{\infty}}\\
       			&=: H_1+H_2.
        	\end{align*}
        	For $H_1$, by $B_{p,1}^{2/p}\hookrightarrow L_{y}^{\infty},$ it follows
        	\begin{align*}
        		H_1&\lesssim \langle \eta-(t+1)v_1\rangle^{l} \sum_{j\geq -1}2^{\frac{2j}{p}}\norm{\Delta_j f(t,\eta,y,v)}_{L_y^p}\\
        		&\lesssim \sum_{j\geq -1}2^{\frac{2j}{p}}\norm{ w(t,x,v)\Delta_j f(t,\eta,y,v)}_{L_y^p} = \norm{f}_{B_{p,1,w}^{2/p}}.
        	\end{align*}
        	For $H_2$, note that for fixed $(t,\eta,v)$, the following embedding holds:
        	\begin{equation}\label{embeding---}
        		\norm{\langle y \rangle^{l} g(t,x,v)}_{L^{\infty}}\lesssim\sum_{j\geq-1} 2^{\frac{2j}{p}}\norm{\langle y \rangle^{l}\Delta_j g(t,x,v)}_{L^p_y},
        	\end{equation}
        	which follows from \cite[pp.156]{et-96}. Let $g(t,\eta,y,v)=f(t,\eta,y+(t+1)v_y,v)$.
        	Recalling the definition of operator $\Delta_j$ in Section \ref{sec-nfs}, we compute
        	\begin{align*}
        		\Delta_j g(y)
        		&=2^{2j}\int_{\mathbb{R}^2} \psi(2^{-j}(y-z))g(t,\eta,z,v)dz\notag\\
        		&= 2^{2j}\int_{\mathbb{R}^2} \psi(2^{-j}(y+(t+1)v_y-z))f(t,\eta,z,v)dz\notag\\
        		&= \Delta_j f(y+(t+1)v_y).
        	\end{align*}
        	Then combining this with
        	\eqref{embeding---}, we obtain
        	\begin{align*}
        		H_2=\norm{\langle y\rangle^{l} g(t,\eta,y,v)}_{L_y^{\infty}}&\lesssim \sum_{j\geq-1} 2^{\frac{2j}{p}}\norm{\langle y \rangle^{l}\Delta_j g(t,\eta,y,v)}_{L^{p}_y}\\
        		   &= \sum_{j\geq-1} 2^{\frac{2j}{p}}\norm{\langle y \rangle^{l}\Delta_j f(t,\eta,y+(t+1)v_y,v)}_{L^{p}_y}\\
        		   &= \sum_{j\geq-1} 2^{\frac{2j}{p}}\norm{\langle y - (t+1)v_y \rangle^{l}\Delta_j f(t,\eta,y,v)}_{L^{p}_y}\\
        		   &\lesssim \sum_{j\geq-1} 2^{\frac{2j}{p}}\norm{w(v)\Delta_j f(t,x,v)}_{L^{p}_y}=\norm{f}_{B_{p,1,w}^{2/p}}.
        	\end{align*}
        This ends the proof of the weighted Besov imbedding \eqref{Embedding inequailty}.
\end{proof}

The following lemma is devoted the bound of the weighted frequency cut-off operators $\Delta_j$ and $S_{j}$.

\begin{lemma}\label{ds-bd-lem}
Let $1\leq p \leq \infty$ and $l\geq0$, it holds that
 \begin{equation}\label{bd - w(v)Delta_k}
		\norm{w(v)\Delta_j f}_{L_y^p}\lesssim \norm{w(v)f}_{L_y^p},\  \norm{w(v)S_{j}f}_{L_y^p} \lesssim \norm{w(v)f}_{L_y^p}.
	\end{equation}
	Furthermore, taking the $L_{\eta}^p$ norm yields:
	$$  \norm{w(v)\Delta_j f}_{L_x^p}\lesssim \norm{w(v)f}_{L_x^p}, \norm{w(v)S_{j}f}_{L_x^p} \lesssim \norm{w(v)f}_{L_x^p}.$$

\end{lemma}
\begin{proof}
We only prove $\eqref{bd - w(v)Delta_k}_1$, as the proof of $\eqref{bd - w(v)Delta_k}_2$ is analogous.
	For $l\geq 0,$  by Young's inequality, a direct computation yields
	\begin{align*}
		\norm{w(v)\Delta_j f}_{L_{y}^p} &= \norm{w(v)\int_{\mathbb{R}^2} 2^{2j}\psi(2^j z)f(t,\eta,y-z,v)dz }_{L_{y}^p}\\
		&\lesssim \norm{\int_{\mathbb{R}^2} 2^{2j} \langle z \rangle^l\psi(2^j z)w(t,\eta,y-z,v)f(y-z)dz }_{L_y^{p}}\\
		&\lesssim\norm{2^{2j} \langle y \rangle^l \psi (2^j y)}_{L_y^1} \norm{w(v)f}_{L_y^p},
	\end{align*}
It is clear that $\norm{2^{2j} \langle y \rangle^l \psi (2^j y)}_{L_y^1}\lesssim \norm{\langle y\rangle^l \psi(y)}_{L^1_y}\lesssim1$, which implies $\eqref{bd - w(v)Delta_k}_1$.
	This ends the proof of Lemma \ref{ds-bd-lem}.
\end{proof}

The following lemma is devoted to the so-called div-curl lemma, which has been used in \cite{wz-24}.
\begin{lemma}[Div-curl lemma]\label{div-curl Lemma}
        Suppose that the functions \( f^{ij} \), for \( i,j = 1,2 \), satisfy
\[
\begin{cases}
\partial_t f^{11} + \partial_x f^{12} = G^1, \\
\partial_t f^{21} - \partial_x f^{22} = G^2,
\end{cases}
\]
and
\[
f^{11}, f^{12}, f^{21}, f^{22} \to 0 \quad \text{as } |x| \to \infty.
\]

Then it holds that
\begin{align*}
\int_0^T \int_{-\infty}^{+\infty} \big( f^{11} f^{22} + f^{12} f^{21} \big) \, dx \, dt
&\le \Bigg| \int_{x<y} f^{11}(0,x) f^{21}(0,y) \, dx \, dy \Bigg|
    \\
&\quad + \Bigg| \int_{x<y} f^{11}(T,x) f^{21}(T,y) \, dx \, dy \Bigg| \\
&\quad + \Bigg| \int_0^T \int_{-\infty}^{+\infty} \Big( \int_{-\infty}^x f^{11}(t,y) \, dy \Big) G^2(t,x) \, dx \, dt \Bigg| \\
&\quad + \Bigg| \int_0^T \int_{-\infty}^{+\infty} \Big( \int_x^{+\infty} f^{21}(t,y) \, dy \Big) G^1(t,x) \, dx \, dt \Bigg|,
\end{align*}
provided that the right-hand side is bounded.
\end{lemma}

	\noindent {\bf Acknowledgements:}
	SQL was supported by grants from the National Natural Science Foundation of China under contracts 12325107 and 12471217.
	YZ was supported by grant from the National Natural Science Foundation of China under contract 12571231.

\medskip
\noindent\textbf{Data Availability Statement:}
Data sharing is not applicable to this article as no datasets were generated or analysed during the current study.

\noindent\textbf{Conflict of Interest:}
The authors declare that they have no conflict of interest.

\end{document}